\documentclass[11pt]{amsart}
\usepackage[margin=1in,foot=30pt]{geometry}
\usepackage{amsmath,amsthm,amssymb}
\usepackage{tcolorbox}
\usepackage{float}
\usepackage{young}
\usepackage[vcentermath]{youngtab}
\newtcolorbox{setting}[2][]{ %colback=white
colframe=c1!35!gray,fonttitle=\bfseries, title=#2,#1}
\numberwithin{equation}{section}
\theoremstyle{definition}
\newtheorem{prop}{\textcolor{dpurb}{Proposition}}
\newtheorem{lemma}[prop]{\textcolor{dpurb}{Lemma}}

\newtheorem{cor}[prop]{\textcolor{dpurb}{Corollary}}

\numberwithin{prop}{section}

\newtheorem{rmk}[prop]{\textcolor{dpurb}{Remark}}

\usepackage{xcolor}
\definecolor{c0}{rgb}{0.0,0.9,0.9}
\definecolor{c1}{rgb}{0.0,0.5,0.6}
\definecolor{c2}{rgb}{0.0,0.0,0.5}
\definecolor{c3}{rgb}{0.5,0.0,0.5}
\definecolor{grn}{rgb}{0,0.4,0}
\definecolor{dgrn}{rgb}{0.0,0.5,0.0}
\definecolor{dblu}{rgb}{0.0,0.1,0.5}
\definecolor{dpur}{rgb}{0.7,0.0,0.7}
\definecolor{dred}{rgb}{0.7,0.0,0.0}
\definecolor{dpurb}{rgb}{0.3,0.0,0.4}
\usepackage[colorlinks=true, linkcolor=c1, citecolor=c3]{hyperref}
\newtheorem{innercustomthm}{\textcolor{dpurb}{\textbf{Theorem}}}
\newenvironment{customthm}[1]
{\renewcommand\theinnercustomthm{#1}\innercustomthm}
  {\endinnercustomthm}
\newtheorem{innercustomcor}{\textcolor{dpurb}{\textbf{Corollary}}}
\newenvironment{customcor}[1]
  {\renewcommand\theinnercustomcor{#1}\innercustomcor}
  {\endinnercustomcor}
\newtheorem{innercustomprop}{\textcolor{dpurb}{\textbf{Proposition}}}
\newenvironment{customprop}[1]
{\renewcommand\theinnercustomprop{#1}\innercustomprop}
  {\endinnercustomprop}
\newtheorem{innercustomlemma}{\textcolor{dpurb}{\textbf{Lemma}}}
\newenvironment{customlemma}[1]
{\renewcommand\theinnercustomlemma{#1}\innercustomlemma}
  {\endinnercustomlemma}
\newcommand{\del}{\partial}

\newcommand{\brs}[1]{\left| #1 \right|}
\newcommand{\brk}[1]{\left[ #1 \right]}
\newcommand{\prs}[1]{\left( #1 \right)}
\newcommand{\sqg}[1]{\left\{ #1 \right\}}
\newcommand{\ip}[1]{\left\langle #1 \right\rangle}

\newcommand{\hsp}{\hspace{0.5cm}}
\newcommand{\N}{\nabla}

\newcommand{\lap}{\Delta}
\newcommand{\gG}{\Gamma}
\renewcommand{\gg}{\gamma}

\newcommand{\gs}{\sigma}

\newcommand{\gw}{\omega}

\renewcommand{\ge}{\epsilon}
\newcommand{\la}{\lambda}

\newcommand{\gY}{\Upsilon}

\newcommand{\w}{\wedge}
\newcommand{\ten}{\otimes}

\newcommand{\dVg}{\dV_{g}}

\newcommand{\dVf}{\dV_{g}^{\phi}}
\newcommand{\dVe}{\dV_{g}^{\eta}}
\DeclareMathOperator{\Rc}{Rc}
\DeclareMathOperator{\Rm}{Rm}
\DeclareMathOperator{\dV}{dV}

\DeclareMathOperator{\tr}{tr}

\DeclareMathOperator{\Vol}{Vol}

\DeclareMathOperator{\supp}{supp}
\DeclareMathOperator{\cP}{P}
\DeclareMathOperator{\cS}{S}
\DeclareMathOperator{\cR}{R}
\DeclareMathOperator{\cQ}{Q}
\DeclareMathOperator{\cW}{W}
\DeclareMathOperator{\cB}{B}
\DeclareMathOperator{\cV}{V}

\usepackage{tcolorbox}
\usepackage{xcolor}
\definecolor{c0}{rgb}{0.0,0.9,0.9}
\newtcolorbox{thmbox}[2][]{colback=white,colframe=c0!35!gray,fonttitle=\bfseries, title=#2,#1}
\begin{document}
\title{Symplectic curvature flow revisted}
\author{Casey Lynn Kelleher}
\address{Department of Mathematics
         Princeton University\\
         Princeton, New Jersey, 08540}
\email{\href{mailto:ckelleher@princeton.edu}{ckelleher@princeton.edu}}
\thanks{The author was supported by a National Science Foundation Mathematical Postdoctoral Research Fellowship}
\date{\today}
\begin{abstract}
We continue studying a parabolic flow of almost K\"{a}hler structures introduced by Streets and Tian which naturally extends K\"{a}hler--Ricci flow onto symplectic manifolds. In the system of primarily the symplectic form, almost complex structure, Chern torsion and Chern connection, we establish new formulas for the evolutions of canonical quantities, in particular those related to the Chern connection. Using this, we give an extended characterization of fixed points of the flow originally performed in \cite{ST}.
\end{abstract}
\maketitle
\section{Introduction}
In \cite{ST} Streets and Tian introduced a curvature flow to investigate the topology and geometry of symplectic manifolds. In this work, we continue their analysis with a focus on further understanding the flow and corresponding evolutions of canonical quantities by providing new formulations and perspectives.\\

We first recall the fundamentals of this particular flow. For a given symplectic manifold $\prs{M^{2n},\gw}$, one can choose a compatible almost complex structure to form an almost K\"{a}hler manifold $(M^{2n}, J, \gw)$. The \emph{symplectic curvature flow} is a coupled degenerate parabolic system given by
\begin{equation*}
\begin{cases}
\tfrac{\del \gw_t}{\del t}= - \cP_t \qquad &\gw_0 = \left. \gw_t \right|_{t=0} \\ 
\tfrac{\del J_t}{\del t} = - \brk{\cP^{2,0+0,2}_t - 2 J_t \Rc_t^{2,0+0,2} }g^{-1}_t  \qquad & J_0 = \left. J_t \right|_{t=0}, 
\end{cases}
\end{equation*}
where $\cP$ is the curvature form of the Hermitian connection on the anticanonical bundle induced by the Chern connection. By Chern--Weil theory $\cP$ is closed and a multiple of the representative of $c_1 \prs{M,J}$. The resultant metric evolution is
\begin{equation}\label{eq:gevolution}
\tfrac{\del g_t}{\del t} = - 2 \Rc_t + 2 \cB_t \qquad g_0 \equiv \left. g_t \right|_{t=0},
\end{equation}
where
\begin{equation}\label{eq:cBforms}
\cB \triangleq \tfrac{1}{4} \cB^1 - \tfrac{1}{2} \cB^2, \qquad \cB^1_{ij}\triangleq g^{kl} g_{mn} \prs{D_i J_k^m}\prs{D_j J_l^n },\qquad \cB^2_{ij} \triangleq g^{kl} g_{mn} \prs{D_k J_i^m} \prs{D_l J_j^n}.
\end{equation}
In \cite{ST}, the authors demonstrate short time existence and characterize the long time existence obstruction in terms of the behavior of the Riemannian tensor (cf. \cite{ST} Theorems 1.6, 1.10) and classify generalizations of fixed points in $\dim_{\mathbb{C}} M = 2$ (cf. Corollary \ref{cor:2.5ofST}).\\
 
 This work proceeds with an underlying theme of reexamining symplectic curvature flow primarily in terms of the Chern connection (Chern derivatives, curvature and torsion).  We provide useful perspectives on identities in K\"{a}hler geometry, particularly along the flow, and showcase naturality and tractability of symplectic curvature flow while setting crucial groundwork for future pursuits.
\subsection{Outline of paper and statement of main results}
In \S \ref{s:background}, after pinning down conventions, we establish curvature identities relating Levi--Civita and Chern quantities. These all are independent of the flow and are generally useful for computations in almost K\"{a}hler geometry. In \S \ref{s:variationformulas}, we establish general variational formulas for curvature and torsion quantities. In \S \ref{s:SCFevs}, we explore variations of canonical objects along the flow. A key consequence of our Chern connection framework is a drastic simplification of Proposition 6.1 of \cite{ST} (reducing the $J$ flow from 11 terms to 1):
\begin{customprop}{A}\label{prop:Jdot} Along symplectic curvature flow the almost complex structure evolves by
\begin{align}
\begin{split}\label{eq:Jev}
\prs{\tfrac{\del J}{\del t}}_i^k &= 4\gw^{re} \prs{\N_r \tau_{ei}^k}.
 \end{split}
\end{align}
\end{customprop}
Interestingly, the quantity by which $J$ flows appears in the variation of $\brs{\tau}^2$ (cf. Lemma \ref{lem:normtauvar}). This yields a clean expression for the evolution of $\brs{\tau}^2$ along symplectic curvature flow (cf. Proposition 7.8 of \cite{ST}).
\begin{customthm}{B}\label{thm:Cherntors}
Along symplectic curvature flow the norm of the Chern torsion tensor evolves by
\begin{align*}
\prs{ \tfrac{\del}{\del t} - \lap }\brs{ \tau }^2 
&=\brs{\cB}^2 - 2 \brs{\N \tau}^2 -  4g^{ap}g^{bq} g^{dr}\Omega_{rpq}^e\prs{ \tau_{edc} \tau_{abc} +  \tau_{cde} \tau_{acb} + \tau_{ced}\tau_{acb}}, 
\end{align*}
and in particular, when $\dim_{\mathbb{C}} M = 2$,
\begin{align*}
\prs{\tfrac{\del}{\del t} - \lap} \brs{\tau}^2&=  - 2 \brs{\N \tau}^2 -  4 g^{dr}g^{ap}g^{bq} \Omega_{rpq}^e\prs{ \tau_{edc} \tau_{abc} +  2 \tau_{ced}\tau_{acb}}.
\end{align*}
\end{customthm}
\begin{customcor}{B}\label{cor:Cherncurvcont}
Let $\prs{M^{4}, \gw_0, J_0}$ be an almost K\"{a}hler manifold. There is a unique solution to symplectic curvature flow on a maximal time interval $[0,\sigma)$. Furthermore, if $\sigma < \infty$ then
\begin{align*}
\limsup_{t \to \sigma} \brs{\Omega}_{C^0} = \infty.
\end{align*}
\end{customcor}
Since the Chern connection naturally ties together $g$ and $J$, one may consider this demonstration of `Chern control' (in contrast to control by the Riemannian curvature in Theorem 1.10 of \cite{ST}) to be more natural. An application of Theorem \ref{thm:Cherntors} lies within the evolution computation of Chern scalar curvature.
\begin{customthm}{C}\label{thm:Chernscalcurv}
Along symplectic curvature flow the Chern scalar curvature evolves by
\begin{align}
\prs{ \tfrac{\del}{\del t} - \lap} \varrho
&= 4 \brs{\Rc- \cB}^2 +2 \lap \brs{\tau}^2 + 4  g^{ip}g^{jq}\prs{ \N_j \N_i \cB_{pq}} +  16 g^{ap} g^{re} g^{bq}g^{rc} \prs{\N_a \N_r \tau_{ebc}} \tau_{pqr}.
\end{align}
\end{customthm}
The variations of $\tau$ and $J$ along the flow play a key role in \S \ref{s:rigidity} regarding our rigidity result, where we build on Corollary 9.5 of \cite{ST} concerning the classification of static structures of the flow (cf. Proposition \ref{thm:rigidity}) by utilizing Sekigawa's formula.
\subsection*{Acknowledgements}
The author thanks Gang Tian and Jeffrey Streets for motivating her to explore this topic and for their endless support. She thanks Yury Ustinovskiy for stimulating conversations.
\section{Background}\label{s:background}
Since an almost K\"{a}hler manifold $\prs{M,g,J}$ is almost Hermitian, we have that the presence of \emph{$g$-compatibility} and a \emph{symplectic form $\gw$}, given respectively by (with local coordinate representation):
\begin{align}
\begin{split}\label{eq:gcompwdefn}
g \prs{X,Y} = g \prs{JX,JY} \qquad & g_{\mathbf{ij}} = J_{\mathbf{i}}^a J_{\mathbf{j}}^bg_{ab},\\
\gw \prs{X,Y} \triangleq g\prs{JX,Y } \qquad &\gw_{\mathbf{ij}} = J_\mathbf{i}^a g_{\mathbf{j}a}.
\end{split}
\end{align}
We adhere to the conventions of Gauduchon and Kobayashi--Nomizu (\cite{KN}, \cite{Gauduchon} pp.259 above (1.1.3)). In coordinates, resultant identities are:
\begin{align}
\begin{split}\label{eq:identitylist}
g_{\mathbf{ij}} = J^s_{\mathbf{i}} \gw_{\mathbf{j}s}& \qquad g^{\mathbf{ij}} = J_s^{\mathbf{i}} \gw^{\mathbf{j}s} \qquad \gw^{\mathbf{ij}} = J^{\mathbf{i}}_s g^{\mathbf{j}s}, \qquad J_{\mathbf{i}}^{\mathbf{j}} = \gw^{\mathbf{j}s} g_{\mathbf{i}s} = \gw_{\mathbf{i}s} g^{\mathbf{j}s} \qquad \gw^{\mathbf{a}c} \gw_{c \mathbf{b}} = \delta^{\mathbf{a}}_{\mathbf{b}}.
\end{split}
\end{align}
Recall the following decomposition of elements of $\prs{T^*M}^{\ten 2}$ into two types:
\begin{align}
\begin{split}\label{eq:typesdecomp}
A_{\mathbf{ij}}^{2,0 + 0,2} &= \tfrac{1}{2} \prs{ A_{\mathbf{ij}} - J_{\mathbf{i}}^sJ_{\mathbf{j}}^u A_{su}}, \\
A_{\mathbf{ij}}^{1,1} \hsp &= A_{\mathbf{ij}} - A_{\mathbf{ij}}^{2,0 + 0,2}  = \tfrac{1}{2} \prs{ A_{\mathbf{ij}} + J_{\mathbf{i}}^sJ_{\mathbf{j}}^u A_{su}}.
\end{split}
\end{align}
We introduce the coordinate expression of the \emph{Nijenhuis tensor}. In accordance with (\cite{KN} pp.123-124), set
\begin{equation}\label{eq:Nijen}
N_{\mathbf{jk}}^{\mathbf{i}} \triangleq 2 \prs{ J_{\mathbf{j}}^p \prs{\del_p J_{\mathbf{k}}^{\mathbf{i}}} - J_{\mathbf{k}}^p \prs{\del_p J_{\mathbf{j}}^{\mathbf{i}}} - J_p^{\mathbf{i}} \prs{\del_{\mathbf{j}} J_{\mathbf{k}}^p} + J_p^{\mathbf{i}} \prs{\del_{\mathbf{k}} J_{\mathbf{j}}^p}}.
\end{equation}
Set $N_{\mathbf{ijk}} \triangleq N_{\mathbf{ij}}^l g_{l\mathbf{k}}$ and note $N$ is type $(3,0+0,3)$, and so $(2,0+0,2)$ in each pair of indices:
\begin{equation}\label{eq:Nsymmetries}
N_{\mathbf{ijk}}= -N_{\mathbf{i}bc} J_{\mathbf{j}}^b J_{\mathbf{k}}^{c} = -J_{\mathbf{i}}^a J_{\mathbf{j}}^b  N_{ab\mathbf{k}}.
\end{equation}
Since $\prs{M,g,J}$ is \emph{almost K\"{a}hler}, we have that $\gw$ is closed.
\begin{equation}\label{eq:AK}
0 = \prs{d \gw}_{\mathbf{ijk}} = \del_{\mathbf{i}} \gw_{\mathbf{jk}} + \del_{\mathbf{j}} \gw_{\mathbf{ki}} + \del_{\mathbf{k}} \gw_{\mathbf{ij}}.
\end{equation}
This small fact is the crucial underpinning of many facts. One consequence is a characterization of the Chern connection for almost K\"{a}hler manifolds as the unique connection $\N = \del + \gY$ such that
\begin{equation}\label{eq:Cherncharac}
\N \gw \equiv 0, \qquad \N J \equiv 0,\qquad \tau_{\N}^{1,1} \equiv 0,
\end{equation}
where $\tau_{\N}$ denotes the torsion tensor of $\N$, a section of $\Lambda^2 \otimes TM$, and $\tau^{1,1}_{\N}$ is the projection of the vector valued torsion two-form onto the space of $\prs{1,1}$-forms. The latter identity expressed in coordinates is
\begin{align*}
\tau_{\mathbf{ij}}^{\mathbf{k}} &\triangleq \gY_{\mathbf{\mathbf{ij}}}^{\mathbf{k}} - \gY_{\mathbf{ji}}^{\mathbf{k}} = - J_{\mathbf{i}}^aJ_{\mathbf{j}}^b \tau_{ab}^{\mathbf{k}}.
\end{align*}
Denote the negative contorsion tensor of the Chern connection (the gap between Levi-Civita and Chern connections) by
\begin{equation}\label{eq:Cherndef}
\Theta_{\mathbf{ij}}^{\mathbf{k}}\triangleq\prs{D - \N}_{\mathbf{ij}}^{\mathbf{k}} =\prs{\gG - \gY}_{\mathbf{ij}}^{\mathbf{k}} = 
-\tfrac{1}{2} \prs{D_{\textbf{i}} J_p^{\textbf{k}}} J_{\textbf{j}}^p.
\end{equation}
From \cite{KN} we may identify $N$ as a Levi-Civita derivative of $\gw$, which yields a characterization of $\Theta$:
\begin{customprop}{4.2 of \cite{KN}}[\textbf{AK}]\label{cor:N=DJ}\footnote{We mark results in the literature which are restricted to the almost K\"{a}hler setting with an `(\textbf{AK})'.}Suppose $\prs{M^{2n},g,J}$ is an almost K\"{a}hler manifold. Then
\begin{equation}\label{eq:N=DJ}
4 \ip{ \prs{D_X J} Y, Z} = \ip{ N \prs{Y, Z}, JX}, \text{ that is, } N_{\mathbf{ij}}^{\mathbf{k}} = 4\gw^{\mathbf{k}p} \prs{D_p \gw_{\mathbf{ij}}} \text{ and } D_{\mathbf{i}} \gw_{\mathbf{j} \mathbf{k}} =  \tfrac{1}{4}  \gw_{\mathbf{i}p} N_{\mathbf{jk}}^p.
\end{equation}
\end{customprop}
\begin{cor}\label{cor:Thetaggtau}From \eqref{eq:Cherndef} it follows that
\begin{equation}\label{cor:1.5}
8 \Theta_{\mathbf{ijk}} = N_{\mathbf{jki}}.
\end{equation}
\begin{proof}
We carefully compute
\begin{align*}
\Theta^{y}_{ni} &= - \tfrac{1}{2}  \prs{D_n   J_p^y} J_i^p  &\mbox{ \eqref{eq:Cherndef}}\\
\Theta_{nim} &= - \tfrac{1}{2}  \prs{D_n   J_p^y g_{ym}} J_i^p  &\mbox{(multiply by $g$)}\\
&= - \tfrac{1}{2}  \prs{D_n  \gw_{pm}} J_i^p  &\mbox{\eqref{eq:gcompwdefn}}\\
&= - \tfrac{1}{8} \gw_{ne} N_{pm}^eJ_i^p &\mbox{Proposition \ref{cor:N=DJ}}\\
&=- \tfrac{1}{8} \prs{J_v^e \gw_{ne} N_{um}^v J_p^u} J_i^p &\mbox{\eqref{eq:Nsymmetries}}\\
&=\tfrac{1}{8} N_{imn},
\end{align*}
relabeling yields the result.
\end{proof}
\end{cor}
This characterization results in an identification of the Nijenhuis tensor with the Chern torsion.
\begin{customthm}{3.4 of \cite{KN}}\label{thm:Nid} If $\prs{M,g,J}$ is almost K\"{a}hler then $N = 8 \tau$.
\end{customthm}
\begin{cor}\label{cor:tausymmetries}
The torsion tensor is of type $(3,0+0,3)$ and thus $(2,0+0,2)$ in pairwise indices, so
\begin{align*}
 \tau_{\mathbf{ij}\mathbf{k}} = - \tau_{\mathbf{ji}\mathbf{k}}, \qquad \tau_{\mathbf{ij}\mathbf{k}} = - \tau_{\mathbf{i}sr} J_{\mathbf{j}}^s J_{\mathbf{k}}^r = - J_{\mathbf{i}}^p J_{\mathbf{j}}^q \tau_{pq\mathbf{k}}.
\end{align*}
\end{cor}
\begin{cor}\label{cor:tautauid} The following contractions are type $(1,1)$:
\begin{align*}
 g^{ab} g^{cd} \tau_{\mathbf{i}ac} \tau_{\mathbf{j}bd} \qquad \qquad g^{ab} g^{cd} \tau_{\mathbf{i}ac} \tau_{\mathbf{j}db} \qquad \qquad g^{ab} g^{cd} \tau_{ac\mathbf{i}} \tau_{bd\mathbf{j}} \\
\gw^{ab} g^{cd} \tau_{\mathbf{i}ac} \tau_{\mathbf{j}bd} \qquad \qquad \gw^{ab} g^{cd} \tau_{\mathbf{i}ac} \tau_{\mathbf{j}db} \qquad \qquad \gw^{ab} g^{cd} \tau_{ac\mathbf{i}} \tau_{bd\mathbf{j}}.\end{align*}
\begin{rmk}
This fact comes from Lemma \ref{cor:tausymmetries} combined with \eqref{eq:gcompwdefn}, and represents all combinations of two Chern torsion tensors in $\prs{T^*M}^{\otimes 2}$.
\end{rmk}
\end{cor}
The combination of Proposition \ref{cor:N=DJ}, Corollary \ref{cor:Thetaggtau} and Theorem \ref{thm:Nid}  yield, roughly speaking, equivalence between $DJ$, $N$, $\Theta$ and $\tau$. We phrase our work primarily in terms of $\tau$.
\begin{prop}
An almost K\"{a}hler manifold such that $\tau_{\N} \equiv 0$ is K\"{a}hler.
\end{prop}
We close this discussion with some useful identities regarding the Chern torsion tensor.
\begin{lemma}\label{lem:tracelesstau} Any trace of $\tau$, with respect to $g$, $\gw$, or $J$, is zero. More precisely
\begin{align*}
g^{ab} \tau_{ab\mathbf{c}} =  g^{bc} \tau_{ab\mathbf{c}} = \gw^{ab} \tau_{ab\mathbf{c}} =\gw^{bc} \tau_{ab\mathbf{c}} = J_a^b \tau_{\mathbf{c}b}^a\equiv 0.
\end{align*}
\begin{proof}This follows from type arguments via \eqref{eq:gcompwdefn} and Corollary \ref{cor:tausymmetries}.
\end{proof}
\end{lemma}
Lastly we state a crucial Bianchi type identity for torsion which will be key to many manipulations.
\begin{customprop}{2 of \cite{Gauduchon}}[\textbf{AK}]\label{prop:taucyclic} 
For an almost K\"{a}hler manifold
\begin{align*}
0 \equiv \tau_{ijk}  +  \tau_{jki} +  \tau_{kij}.
\end{align*}
\end{customprop}
\subsection{Identities of $\mathbf{\cB}$ terms}
\noindent Using the preliminaries above, we record Chern expressions of the $\cB$ terms featured in \eqref{eq:cBforms} from \cite{ST}, as they are prevalent throughout our work.
\begin{lemma}\label{lem:Bformulas}
We have that
\begin{align*}
\cB^1_{ij} =4 g^{kl} g^{wv} \tau_{vki} \tau_{wlj} , \quad \cB^2_{ij} = 4 g_{mn} g^{wv} \tau_{iv}^n \tau_{jw}^m, \quad \cB_{ij} =- 2 \tau_{ik}^w \tau_{jw}^k.
\end{align*}
\begin{proof}
We will first manipulate $\cB^2$. In this case we have that
\begin{align*}
\cB^2_{ij} &= g^{kl} g_{mn} \prs{D_k J_i^m} \prs{D_l J_j^n} &\mbox{\eqref{eq:cBforms}}\\
&= g^{kl} g_{mn} \prs{2 \gw^{vm} \tau_{vik}}\prs{2 \gw^{wn} \tau_{wjl}} &\mbox{Corollary \ref{cor:N=DJ}} \\
&= 4 g^{kl} g^{wv} \tau_{vik} \tau_{wjl} \hfill &\mbox{\eqref{eq:gcompwdefn}} \\
&= 4 g_{mn} g^{wv} \tau_{iv}^n \tau_{jw}^m.
\end{align*}
Next we manipulate $\cB^1$, first generating the expression described above and then characterizing $\cB$.
\begin{align*}
\tfrac{1}{2} \cB^1_{ij} &= \tfrac{1}{2} g^{kl} g_{mn} \prs{D_i J_k^m}\prs{D_j J_l^m } &\mbox{\eqref{eq:cBforms}}\\
&= \tfrac{1}{2} g^{kl} g_{mn} \prs{2 \gw^{vm} \tau_{vki}}\prs{2 \gw^{wn} \tau_{wlj}} &\mbox{Proposition \ref{cor:N=DJ}}\\
&=2 g^{kl} g^{wv} \tau_{vki} \tau_{wlj}\\
&= -2 g^{kl} g^{wv} \prs{\tau_{kiv} + \tau_{ivk} } \tau_{wlj} &\mbox{Proposition \ref{prop:taucyclic}} \\
&= 4 g^{kl} \tau_{ik}^w \tau_{wlj} \\
&= -4 g^{kl} \tau_{ik}^w \prs{\tau_{ljw} + \tau_{jwl} }&\mbox{Proposition \ref{prop:taucyclic}} \\
&= \cB^2_{ij}  - 4 \tau_{ik}^w \tau_{jw}^k. &\mbox{\eqref{eq:cBforms}}
\end{align*}
Therefore
\begin{equation*}
\tfrac{1}{2} \cB^1_{ij} - \cB^2_{ij} =   - 4 \tau_{ik}^w \tau_{jw}^k \triangleq 2\cB_{ij},
\end{equation*}
which concludes the result.
\end{proof}
\end{lemma}
\begin{cor}\label{cor:Btraces}
We have that
\begin{equation*}
\brs{\tau}^2 = \tfrac{1}{4} g^{ij} \cB^1_{ij} = \tfrac{1}{4} g^{ij}\cB^2_{ij} = - \tfrac{1}{2} g^{ij}\cB_{ij}.
\end{equation*}
\end{cor}
We lastly recall a result which demonstrates the convenient structure of $\cB$ terms in $\dim_{\mathbb{C}}(M) = 2$.
\begin{customlemma}{4.10 of \cite{Dai1}}\label{lem:Dai14.10}
Suppose $\prs{M^4,J,\gw}$ is almost K\"{a}hler. There exists a local unitary frame such that
\begin{align*}
g = \begin{pmatrix}
1 && & \\
&1 & & \\
&& 1&\\
&&& 1
\end{pmatrix}\text{ and }J = \begin{pmatrix}
 & 1& & \\
-1 & & & \\
&& &1\\
&&-1 & 
\end{pmatrix},
\end{align*}
and furthermore in such coordinates
\begin{align*}
\cB^1 = 
\begin{pmatrix}
2 \brs{\tau}^2 && & \\
& 2 \brs{\tau}^2 & & \\
&& 0&\\
&&& 0
\end{pmatrix}, \qquad
\cB^2 = 
\begin{pmatrix}
 \brs{\tau}^2 && & \\
&  \brs{\tau}^2 & & \\
&&\brs{\tau}^2 &\\
&&&\brs{\tau}^2
\end{pmatrix}.
\end{align*}
\end{customlemma}
\subsection{Chern curvature identities}
Here we record some main identities concerning the Chern connection's associated curvature $\Omega_{\N}$ and torsion $\tau_{\N}$.
\begin{lemma}\label{lem:Cherncurvsym}
Let $\prs{M^{2n},g,J}$ be almost Hermitian. Then
\begin{align*}
\Omega_{\mathbf{ijkl}} = - \Omega_{\mathbf{jikl}} = - \Omega_{\mathbf{ijlk}} = \Omega_{\mathbf{ij}ab} J^a_{\mathbf{k}} J^b_{\mathbf{l}} .
\end{align*}
\end{lemma}
\noindent Aspects of the Riemann curvature tensor translate with some residual torsion terms. Define \emph{Chern--Ricci curvature} ($\cP$), \emph{twisted Chern--Ricci curvature} ($\cS$), and \emph{twisted Ricci curvature} ($\cQ$, often referred to as $\Rc (\gw)$) by
\begin{align*}
\cP_{\mathbf{ab}} &\triangleq  \gw^{cd} \Omega_{\mathbf{ab}cd} = J_c^e J_d^f \gw^{cd}   \Omega_{\mathbf{ab}ef} =  J_c^e  g^{cf}   \Omega_{\mathbf{ab}ef} =  J_c^e  \Omega_{\mathbf{ab}e}^c \\
\cS_{\mathbf{cd}}  &\triangleq  \gw^{ab} \Omega_{ab\mathbf{cd}}, \qquad \cQ_{\mathbf{cd}} \triangleq  \gw^{ab} \Rm_{ab\mathbf{cd}} \qquad \cV_{\mathbf{ab}}\triangleq  \Omega_{r \mathbf{ab}}^r,
\end{align*}
and finally, the \emph{Chern scalar curvature} is
\begin{equation*}
\varrho \triangleq  \gw^{ba} \cP_{ab} = \gw^{dc} \cS_{cd}.
\end{equation*}
This differs from \emph{Riemannian scalar curvature} $\cR$ by a multiple of $|\tau|^2$ (cf. Corollary \ref{cor:scalcurv}).
\subsection{Analogues of Riemannian symmetries}
We compute parallel identities to the Riemannian case using the Chern curvature, keeping track of the torsion tensor quantities.
\begin{lemma}\label{lem:commutator}
We have that
\begin{align*}
\brk{\N_a,\N_b} A_{q_1 \cdots q_n}^{p_1 \cdots p_n} =\Sigma_{i=1}^n \Omega^{p_i}_{abd} A^{p_1 \cdots \widehat{p_i} d \cdots p_n}_{q_1 \cdots q_n} - \Sigma_{j=1}^m \Omega_{abq_j}^d A^{p_1 \cdots p_n}_{q_1 \cdots \widehat{q_j} d \cdots q_n} - \tau_{ab}^e \prs{ \N_e A_{q_1 \cdots q_n}^{p_1 \cdots p_n}},
\end{align*}
where here $\hat{p}_i, \hat{q}_j$ denotes the excision of these indices.
\end{lemma}
\begin{customthm}{III.5.3 of \cite{KN}}\label{thm:CBianchi}The following identities hold.
\begin{align}
\Omega_{abc}^k + \Omega_{cab}^k + \Omega_{bca}^k &=   \prs{\N_a \tau_{bc}^k + \N_b \tau_{ca}^k + \N_c \tau_{ab}^k } - \prs{\tau_{as}^k \tau_{bc}^s +\tau_{cs}^k \tau_{ab}^s+ \tau_{bs}^k \tau_{ca}^s}, \label{eq:curve1}\\
\N_a \Omega_{bcj}^k + \N_b \Omega_{caj}^k +  \N_c \Omega_{abj}^k &=  \Omega_{aij}^k \tau_{bc}^i + \Omega_{bij}^k \tau_{ca}^i+ \Omega_{cij}^k \tau_{ab}^i.\label{eq:curve2}
%\N_a \Omega_{jkbc} + \N_b \Omega_{jkca} +  \N_c \Omega_{jkab} &= \textcolor{red}{(finish)}.\label{eq:curve3}
\end{align}
\end{customthm}
\begin{cor}\label{cor:PBianchi}
By tracing through \eqref{eq:curve2}, following identities hold.
\begin{align*}
\N_a \cP_{bc} + \N_b \cP_{ca} +  \N_c \cP_{ab} &=  \cP_{ai} \tau_{bc}^i + \cP_{bi} \tau_{ca}^i+ \cP_{ci} \tau_{ab}^i.
%\N_a \cS_{j}^k + 2 \gw^{bc} \N_b \Omega_{caj}^k  &=  2 \gw^{bc} \Omega_{bij}^k \tau_{ca}^i\\
%\N_c \varrho &= 2 \gw^{ab} \prs{ \N_a \cP_{bc}  -  \cP_{ai} \tau_{bc}^i }.
\end{align*}
%
\begin{comment}
\begin{proof}
These follow by taking traces by $J_j^k$, $\gw^{bc}$, and $J_j^k \gw_{ab}$ on the second identity of Theorem \ref{thm:CBianchi}.
\end{proof}
\end{comment}
\end{cor}
\noindent The next identity is another 'translation' of a Riemannian curvature symmetry to Chern curvature. This identity is crucial for computing remarkably clean identities curvature quantities.
\begin{lemma}\label{lem:Chernsymm} The following holds:
\begin{align}
\begin{split}\label{eq:Tdefn}
 \Omega_{bdac}- \Omega_{acbd} 
&= \prs{ \N_a \tau_{bdc} + \N_b \tau_{cad} + \N_c \tau_{dba} + \N_d \tau_{acb} } \\
&\hsp- \prs{  \tau_{asd}\tau_{bc}^s + \tau_{bsc} \tau_{da}^s + \tau_{dsc}\tau_{ab}^s + \tau_{asb} \tau_{cd}^s }\\
&\triangleq  T_{bdac},
\end{split}
\end{align}
which implies the following symmetries of $T$:
\begin{equation}\label{eq:Tsyms}
T_{abij} = - T_{baij}=  - T_{abji}, \qquad T_{abij} = -T_{ijab}.
\end{equation}
\begin{proof}
Starting from Theorem \ref{thm:CBianchi} we write out
\begin{align*}
\Omega_{abcd} + \Omega_{cabd} + \Omega_{bcad} &=   \prs{\N_a \tau_{bcd} + \N_b \tau_{cad} + \N_c \tau_{abd} } - \prs{\tau_{asd} \tau_{bc}^s +\tau_{csd} \tau_{ab}^s+ \tau_{bsd} \tau_{ca}^s} \\
\Omega_{dabc} + \Omega_{abdc} + \Omega_{bdac} &=   \prs{\N_d \tau_{abc} + \N_a \tau_{bdc} + \N_b \tau_{dac} } - \prs{\tau_{dsc} \tau_{ab}^s +\tau_{asc} \tau_{bd}^s+ \tau_{bsc} \tau_{da}^s} \\
\Omega_{cdab} + \Omega_{dacb} + \Omega_{acdb}&=   \prs{\N_c \tau_{dab} + \N_d \tau_{acb} + \N_a \tau_{cdb} } - \prs{\tau_{csb} \tau_{da}^s +\tau_{dsb} \tau_{ac}^s+ \tau_{asb} \tau_{cd}^s} \\
\Omega_{bcda} + \Omega_{dbca} + \Omega_{cdba}&=   \prs{\N_b \tau_{cda} + \N_c \tau_{dba} + \N_d \tau_{bca} } - \prs{\tau_{bsa} \tau_{cd}^s +\tau_{csa} \tau_{db}^s+ \tau_{dsa} \tau_{bc}^s}.
\end{align*}
We sum together each line. First we approach the higher order terms
\begin{align*}
\Sigma_{i=1}^{12} \prs{Q^i}_{abcd} &=  \N_a \tau_{bcd} + \N_b \tau_{cad}+ \N_c \tau_{abd} + \N_d \tau_{abc} + \N_a \tau_{bdc} + \N_b \tau_{dac} \\
& + \N_c \tau_{dab} + \N_d \tau_{acb} + \N_a \tau_{cdb} + \N_b \tau_{cda} + \N_c \tau_{dba} + \N_d \tau_{bca}.
\end{align*}
Via Proposition \ref{prop:taucyclic}\begin{equation*}
Q^1 + Q^9 = Q^{5}, \quad Q^6 + Q^{10} = Q^2, \quad Q^3 + Q^7 = Q^{11}, \quad Q^4 + Q^{12} = Q^8.
\end{equation*}
Thus all that remains is
\begin{align*}
\Sigma_{i=1}^{12} \prs{Q^i}_{abcd}  = 2 \prs{ \N_a \tau_{bdc} + \N_b \tau_{cad} + \N_c \tau_{dba} + \N_d \tau_{acb} }.
\end{align*}
Likewise for the quadratic $\tau$ terms we have that
\begin{align*}
\Sigma_{i=1}^{12} \prs{R^i}_{abcd} &= - \prs{\tau_{asd} \tau_{bc}^s +\tau_{csd} \tau_{ab}^s+ \tau_{bsd} \tau_{ca}^s} - \prs{\tau_{dsc} \tau_{ab}^s +\tau_{asc} \tau_{bd}^s+ \tau_{bsc}\tau_{da}^s}\\
&- \prs{\tau_{csb} \tau_{da}^s +\tau_{dsb} \tau_{ac}^s+ \tau_{asb} \tau_{cd}^s} - \prs{\tau_{bsa} \tau_{cd}^s +\tau_{csa} \tau_{db}^s+ \tau_{dsa} \tau_{bc}^s}.
\end{align*}
We manipulate the following four terms using Proposition \ref{prop:taucyclic} and relabelling.

\begin{minipage}[t]{0.5\textwidth}
\begin{align*}
\prs{R^5 + R^{11}}_{abcd} &=- \prs{ \tau_{asc} - \tau_{csa} }\tau_{bd}^s \\
&=- \prs{ \tau_{asc} + \tau_{sca} }\tau_{bd}^s \\
&= \tau_{cas}\tau_{bd}^s.
\end{align*}
\end{minipage}
\begin{minipage}[t]{0.5\textwidth}
\begin{align*}
\prs{R^3 + R^8}_{abcd} &= - \prs{\tau_{bsd} - \tau_{dsb} } \tau_{ca}^s \\
&= - \prs{\tau_{bsd} + \tau_{sdb} } \tau_{ca}^s \\
&=\tau_{dbs}\tau_{ca}^s.
\end{align*}
\end{minipage}
It follows that
\begin{equation*}
0 \equiv R^3 + R^5 +R^8 + R^{11}.
\end{equation*}
We manipulate the remaining terms using Proposition \ref{prop:taucyclic} and Corollary \ref{cor:tausymmetries},

\begin{minipage}[t]{0.5\textwidth}
\begin{align*}
\prs{R^1 + R^{12}}_{abcd} &=- \prs{\tau_{asd} + \tau_{dsa}}\tau_{bc}^s \\
&= - \prs{- \tau_{sad} + \tau_{dsa}}\tau_{bc}^s \\
&=- \prs{2 \tau_{dsa} + \tau_{ads}}\tau_{bc}^s \\
&= - 2 \tau_{dsa}\tau_{bc}^s -  \tau_{ads}\tau_{bc}^s.
\end{align*}
\begin{align*}
\prs{R^2 + R^4}_{abcd} &=- \prs{ \tau_{csd} +  \tau_{dsc} }\tau_{ab}^s \\
&=- \prs{ \tau_{csd} -  \tau_{sdc} }\tau_{ab}^s \\
&=- \prs{ 2\tau_{csd} +  \tau_{dcs} }\tau_{ab}^s \\
&=- 2 \tau_{csd} \tau_{ab}^s -  \tau_{dcs} \tau_{ab}^s.
\end{align*}
\end{minipage}
\begin{minipage}[t]{0.5\textwidth}
\begin{align*}
\prs{R^6 + R^{7}}_{abcd} &= - \prs{\tau_{bsc} + \tau_{csb}} \tau_{da}^s\\
&= - \prs{\tau_{bsc} - \tau_{scb}} \tau_{da}^s\\
&= - \prs{2\tau_{bsc} + \tau_{cbs} } \tau_{da}^s\\
&= - 2 \tau_{bsc} \tau_{da}^s - \tau_{cbs}  \tau_{da}^s .
\end{align*}
\begin{align*}
\prs{R^9 + R^{10}}_{abcd} &= - \prs{ \tau_{asb} + \tau_{bsa} } \tau_{cd}^s \\
&= - \prs{ \tau_{asb} - \tau_{sba} } \tau_{cd}^s \\ 
&= - \prs{2 \tau_{asb}  + \tau_{bas} } \tau_{cd}^s \\
&=- 2  \tau_{asb} \tau_{cd}^s -  \tau_{bas} \tau_{cd}^s.
\end{align*}
\end{minipage}

Summing these up and rearranging accordingly yields
\begin{align*}
- \tfrac{1}{2} \Sigma_{i=1}^{12} \prs{R^i }_{abcd}
&=  \prs{\tau_{dsa}\tau_{bc}^s + \tau_{bsc} \tau_{da}^s + \tau_{csd} \tau_{ab}^s + \tau_{asb} \tau_{cd}^s }\\
&\hsp + \prs{\tau_{ads}\tau_{bc}^s +  \tau_{dcs} \tau_{ab}^s} \\
&= \prs{  \tau_{asd}\tau_{bc}^s + \tau_{bsc} \tau_{da}^s + \tau_{dsc}\tau_{ab}^s + \tau_{asb} \tau_{cd}^s }, &\mbox{Proposition \ref{prop:taucyclic}}
\end{align*}
which concludes the result.
\end{proof}
\end{lemma}
\subsection{Riemannian curvature conversion}
We relate Riemannian curvature quantities to Chern connection counterparts.
\begin{lemma}\label{lem:RmasOmega}
We have that
\begin{align}
\begin{split}\label{eq:Rm}
\Rm_{ijkl} &= \Omega_{ijkl} + \prs{\N_i \tau_{klj}} - \prs{\N_j \tau_{kli}}+  \tau_{sli}\prs{ g^{sd} \tau_{kdj}} - \tau_{slj}\prs{g^{sd} \tau_{kdi}}  +  \tau_{ij}^c\tau_{klc}.
\end{split}
\end{align}
\begin{proof} We compute, whittling down to connection coefficients
\begin{align*}
\Rm_{ijk}^l &= \del_i \gG_{jk}^l -  \del_j \gG_{ik}^l + \gG_{is}^l \gG_{jk}^s -  \gG_{js}^l \gG_{ik}^s \\
&= \prs{\del_i \gY_{jk}^l -  \del_j \gY_{ik}^l + \gY_{is}^l \gY_{jk}^s -  \gY_{js}^l \gY_{ik}^s } + \prs{\del_i \Theta_{jk}^l -  \del_j \Theta_{ik}^l +  \Theta_{is}^l \Theta_{jk}^s -  \Theta_{js}^l \Theta_{ik}^s }\\
& \hsp + \gY_{is}^l \Theta_{jk}^s -  \gY_{js}^l \Theta_{ik}^s + \Theta_{is}^l \gY_{jk}^s -  \Theta_{js}^l \gY_{ik}^s. &\mbox{\eqref{eq:Cherndef}}
\end{align*}
We manipulate two terms
\begin{align*}
\del_i \Theta_{jk}^l - \del_j \Theta_{ik}^l &= \prs{\N_i \Theta_{jk}^l + \gY_{ij}^c \Theta_{ck}^l + \gY_{ik}^c \Theta_{jc}^l - \gY_{ic}^l \Theta_{jk}^c } - \prs{\N_j \Theta_{ik}^l + \gY_{ji}^c \Theta_{ck}^l + \gY_{jk}^c \Theta_{ic}^l - \gY_{jc}^l \Theta_{ik}^c } \\
&= \prs{\N_i \Theta_{jk}^l - \N_j \Theta_{ik}^l} + \tau_{ij}^c \Theta_{ck}^l - \prs{\gY_{jk}^c \Theta_{ic}^l - \gY_{ik}^c \Theta_{jc}^l  + \gY_{ic}^l \Theta_{jk}^c - \gY_{jc}^l \Theta_{ik}^c}.
\end{align*}
Inserting this into our expression for $\Rm$ it follows that
\begin{align}
\begin{split}\label{eq:RmOmega}
\Rm_{ijk}^l &= \Omega_{ijk}^l +  \N_i \Theta_{jk}^l - \N_j \Theta_{ik}^l + \Theta_{is}^l \Theta_{jk}^s - \Theta_{js}^l \Theta_{ik}^s + \tau_{ij}^c \Theta_{ck}^l\\
&= \Omega_{ijk}^l + g^{lq}\prs{\N_i \tau_{kqj}} -g^{lq} \prs{\N_j \tau_{kqi}}  + \prs{ g^{lq} \tau_{sqi} }\prs{ g^{sd} \tau_{kdj}} - \prs{g^{lq} \tau_{sqj} }\prs{g^{sd} \tau_{kdi}}\\
&\hsp  + g^{lq} g_{cp} \tau_{kq}^p \tau_{ij}^c.
\end{split}
\end{align}
Lastly, lowering the last index by multiplication by $g$ yields the result.
\end{proof}
\end{lemma}
%
\begin{comment}
\begin{cor}\label{cor:Rmtypelastind}
We have that
%
\begin{align*}
\tfrac{1}{2} \prs{\Rm_{ijkl} - J_k^a J_l^b\Rm_{ijab}} &= \prs{\N_i \tau_{klj}} - \prs{\N_j \tau_{kli}}  +\tau_{klc} \tau_{ij}^c\\
\tfrac{1}{2} \prs{ \Rm_{ijkl} +  J_k^a J_l^b\Rm_{ijab}}&= \Omega_{ijkl} + \tau_{sli}\prs{ g^{sd} \tau_{kdj}} - \tau_{slj}\prs{g^{sd} \tau_{kdi}}.
\end{align*}
%
\end{cor}
\end{comment}
%
\begin{cor}\label{cor:RcOmega}
By tracing through Lemma \ref{lem:RmasOmega},
\begin{align*}
\Rc_{jk} &= \cV_{jk} + g^{il}\prs{\N_i \tau_{klj}} - g^{sd}\tau_{slj} \tau_{kd}^l  +  g^{li}\tau_{ij}^c\tau_{klc}.
\end{align*}
\end{cor}
\begin{lemma}\label{lem:Omegatraced}
We have that
\begin{equation*}\label{eq:Omegatraced}
\cV_{jk} = \tfrac{1}{2} J_k^a \cP_{ja}-g^{ur}\prs{\N_r \tau_{ukj}} - \tau_{je}^s \tau_{ks}^e .
\end{equation*}
\begin{proof}
We manipulate
\begin{align*}\label{eq:RcP1}
\cV_{jk} = g^{re} \Omega_{rjke} &= J_k^a J_e^b g^{re} \Omega_{rjab} &\mbox{Lemma \ref{lem:Cherncurvsym}}\\
&= J_k^a \gw^{br} \Omega_{rjab}\\
&=- J_k^a \gw^{br} \prs{ \Omega_{jarb} + \Omega_{arjb} }\\
&\hsp + J_k^a \gw^{br} \prs{ \N_j \tau_{arb} + \N_a \tau_{rjb} + \N_r \tau_{jab} }\\
&\hsp -J_k^a \gw^{br} g_{bs} \prs{\tau_{re}^s \tau_{ja}^e + \tau_{je}^s \tau_{ar}^e  + \tau_{ae}^s \tau_{rj}^e } &\mbox{Theorem \ref{thm:CBianchi}}\\
&= J_k^a \cP_{ja} - J_k^a \gw^{br}  \Omega_{arjb} + J_k^a \gw^{br} \prs{\N_r \tau_{jab} } \\
&\hsp + J_k^a J^r_s \prs{ \tau_{je}^s \tau_{ar}^e  + \tau_{ae}^s \tau_{rj}^e } \\
&=   J_k^a \cP_{ja} - J_k^a \gw^{br}  \Omega_{arjb} - \prs{\N_r \tau_{jk}^r }- 2\tau_{je}^s \tau_{ks}^e . &\mbox{Corollary \ref{cor:tausymmetries}} 
\end{align*}
Note that for the second term we have that
\begin{align*}
- J_k^a \gw^{br}  \Omega_{arjb} &= - J_k^a \gw^{br}  T_{arjb}- J_k^a \gw^{br}  \Omega_{jbar} &\mbox{Lemma \ref{lem:Chernsymm}} \\
\hsp &=  - J_k^a \gw^{br}  T_{arjb}- J_k^a \gw^{br}  \Omega_{rjab}.
\end{align*}
Now we have that
\begin{equation}\label{eq:gOJP}
\cV_{jk} = \tfrac{1}{2} J_k^a \cP_{ja} - \tfrac{1}{2} J_k^a \gw^{br} T_{arjb} - \tfrac{1}{2} \prs{\N_r \tau_{jk}^r }- \tau_{je}^s \tau_{ks}^e .
\end{equation}
Now we compute out
\begin{align*}
\gw^{br}T_{arjb}&= \gw^{br} \prs{ \N_b \tau_{raj} + \N_b \tau_{rja} + \N_j \tau_{arb}-\N_a \tau_{jbr} }\\
&\hsp -\gw^{br}  \prs{ g_{am}\tau_{js}^m\tau_{br}^s+  g_{jm} \tau_{rs}^m \tau_{ab}^s  + g_{rm} \tau_{bs}^m \tau_{ja}^s + g_{bm} \tau_{as}^m \tau_{rj}^s}\\
&\hsp  -\gw^{br}   g_{sm} \prs{\tau_{aj}^s\tau_{br}^m+  \tau_{ab}^s \tau_{jr}^m} &\mbox{Lemma \ref{lem:Chernsymm}}\\
%%%%
&=\gw^{br} \prs{ \N_b \tau_{raj} - \N_b \tau_{jra} }-\gw^{br}  \prs{ g_{jm} \tau_{rs}^m \tau_{ab}^s  + g_{bm} \tau_{as}^m \tau_{rj}^s}\\
&\hsp  -\gw^{br}   g_{sm} \tau_{ab}^s \tau_{jr}^m. &\mbox{Proposition \ref{lem:tracelesstau}}
 \end{align*}
We first simplify the lower order terms.
\begin{align*}
\brk{\gw^{br} T_{arjb} }_{\tau \ast \tau} &= -\gw^{br} \prs{ g_{jm} \tau_{rs}^m \tau_{ab}^s  + g_{bm} \tau_{as}^m \tau_{rj}^s +  g_{sm} \tau_{ab}^s \tau_{jr}^m}\\
&= -\gw^{br} \prs{ \tau_{rsj} \tau_{ab}^s  + \tau_{asb} \tau_{rj}^s + \tau_{abm} \tau_{jr}^m}\\
&= -\gw^{br} \prs{ - \tau_{sjr} \tau_{ab}^s - \tau_{jrs} \tau_{ab}^s  + \tau_{asb} \tau_{rj}^s + \tau_{abm} \tau_{jr}^m} &\mbox{Proposition \ref{prop:taucyclic}}\\
&= -\gw^{br} \prs{ - \tau_{sjr} \tau_{ab}^s   + \tau_{asb} \tau_{rj}^s }\\
&=\gw^{br} \tau_{sjr} \tau_{ab}^s   -\gw^{br} g^{sc}\tau_{asb} \tau_{rjc}\\
&=\gw^{br} \tau_{sjr} \tau_{ab}^s   + J_r^y\gw^{br}J_c^z g^{sc}\tau_{asb} \tau_{yjz} &\mbox{Corollary \ref{cor:tausymmetries}} \\
&=\gw^{br} \tau_{sjr} \tau_{ab}^s   +\gw^{zs}\tau_{as}^y \tau_{yjz}\\
&=\gw^{br} \tau_{sjr} \tau_{ab}^s   +\gw^{rb}\tau_{ab}^s \tau_{sjr}. &\mbox{(relabel)}
\end{align*}
Thus it follows that
\begin{align*}
-  J_k^a\gw^{br}&T_{arjb}  -  \prs{\N_r \tau_{jk}^r } \\&= - J_k^a \gw^{br} \prs{ \N_b \tau_{raj} + \N_b \tau_{rja} }- \prs{\N_r \tau_{jk}^r } \\
&= - J_k^a \gw^{br} \prs{\N_b \tau_{raj}} - J_k^a \gw^{br} \prs{\N_b \tau_{rja}} - g^{ru} \prs{\N_r \tau_{jku}}\\
&=  g^{ur}\prs{\N_r \tau_{kuj}} -g^{ur} \prs{\N_r \tau_{ujk}} - g^{ru} \prs{\N_r \tau_{jku}} &\mbox{Corollary \ref{cor:tausymmetries}} \\
&=  -2g^{ur}\prs{\N_r \tau_{ukj}}.
\end{align*}
Inserting this into \eqref{eq:gOJP} yields the result.
\end{proof}
\end{lemma}
\begin{cor}\label{cor:Rcid}
We have that
\begin{align*}
\Rc_{jk} 
&=\tfrac{1}{2} J_k^a \cP_{ja}  -2g^{ur}\prs{\N_r \tau_{ukj}}-2 \tau_{je}^s \tau_{ks}^e.
\end{align*}
\begin{proof}
Inserting the result of Lemma \ref{lem:Omegatraced} into Corollary \ref{cor:RcOmega} yields
\begin{align*}
\Rc_{jk} 
&=\tfrac{1}{2} J_k^a \cP_{ja}-2g^{ur}\prs{\N_r \tau_{ukj}}- \tau_{je}^s \tau_{ks}^e  - g_{jp} g^{sd} \tau_{se}^p \tau^e_{kd}  - g^{mn} g_{cp} \tau_{jm}^c \tau_{kn}^p.
\end{align*}
We perform one more manipulation to our lower order terms,
\begin{align*}
- g_{jp} g^{sd} \tau_{se}^p \tau_{kd}^e &=-  g^{sd} \tau_{sej} \tau_{kd}^e \\
 &=  g^{sd} \prs{ \tau_{ejs} +  \tau_{jse} } \tau_{kd}^e. &\mbox{Proposition \ref{prop:taucyclic}}
\end{align*}
Inserting this in yields the result.
\end{proof}
\end{cor}
\begin{cor}\label{cor:scalcurv} We have that $\cR = \tfrac{1}{2} \varrho -  \brs{ \tau }^2$.
\begin{proof}
Trace Corollary \ref{cor:Rcid} and apply Corollary \ref{cor:Btraces} to the last term.
\end{proof}
\end{cor}
%
%\begin{cor}\label{eq:scrRexp}
%We have that
%%
%\begin{align*}
%\mathcal{R}_{jl} 
%&= \gw^{rk} \prs{\N_r \tau_{kjl} + \N_r \tau_{klj} }.
%\end{align*}
%%
%\begin{proof}
%This follows immediately by inserting Corollary \ref{cor:Rc2002} for a formula for the $(2,0+0,2)$ part of $\Rc$ into the definition of $\mathcal{R}$, \eqref{eq:mathcalR}.
%\end{proof}
%\end{cor}
%
\begin{cor}\label{cor:KI} We have that
\begin{equation*}
\cB_{ij} = \Rc_{ij}^{1,1} - \tfrac{1}{2} J_i^s \cP^{1,1}_{js}.
\end{equation*}
\begin{proof} This follows by projecting Corollary \ref{cor:Rcid} onto the $(1,1)$ part.
%We simply compute:
%%
%\begin{align*}
%2 \prs{\Rc_{jk}}^{1,1}
%&=\tfrac{1}{2} J_k^e \cP_{je}  -2g^{ur}\prs{\N_r \tau_{ukj}}-2 \tau_{je}^s \tau_{ks}^e \\
%&\hsp +J_j^a J_k^b \prs{ \tfrac{1}{2} J_b^e \cP_{ae}  -2g^{ur}\prs{\N_r \tau_{uba}}-2 \tau_{ae}^s \tau_{bs}^e}&\mbox{Corollary \ref{cor:Rcid}}\\
%%&=\tfrac{1}{2} J_k^e \cP_{je}  -2g^{ur}\prs{\N_r \tau_{ukj}}-2 \tau_{je}^s \tau_{ks}^e \\
%%&\hsp + \prs{- \tfrac{1}{2}J_j^a \cP_{ak}  +2g^{ur} \prs{\N_r \tau_{ukj}}-2 J_j^a J_k^b \tau_{ae}^s \tau_{bs}^e} &\mbox{simplify}\\
%&=\tfrac{1}{2} \prs{J_k^e \cP_{je}+J_j^a \cP_{ka} }-2 \tau_{je}^s \tau_{ks}^e -2 \prs{J_e^a\tau_{ja}^s} \prs{-J_w^e\tau_{ks}^w}&\mbox{simplify} \\
%&=\tfrac{1}{2}J_k^e \prs{ \cP_{je}-J_j^a J_e^u \cP_{ua} }-4 \tau_{je}^s \tau_{ks}^e &\mbox{Lemma \ref{cor:tausymmetries}}\\
%&=\tfrac{1}{2}J_k^e \prs{ \cP_{je} + J_j^a J_e^u \cP_{au} }-4 \tau_{je}^s \tau_{ks}^e &\mbox{$\cP$ skew}\\
%&=J_k^e  \cP_{je}^{1,1} + 2 \cB_{jk}, &\mbox{Lemma \ref{lem:Bformulas}}%-4 \tau_{je}^s \tau_{ks}^e,
%\end{align*}
%as desired.
\end{proof}
\end{cor}
\begin{lemma}\label{lem:RcOmega} 
We have that
\begin{align*}
\cV_{ab}^{2,0+0,2}=- g^{mq} \prs{\N_m \tau_{qab}}
\end{align*}
\begin{proof}
With this identity  for $T$ above we also have that
\begin{align*}
\cV_{ab} &= g^{er} \Omega_{rabe} \\
&= g^{er} J_b^p J_e^q \Omega_{rapq}&\mbox{Lemma \ref{lem:Cherncurvsym}} \\
&=g^{er} J_b^p J_e^q \Omega_{pqra} + g^{er} J_b^p J_e^q\prs{T_{rapq}}&\mbox{Lemma \ref{lem:Chernsymm}} \\
&=g^{er} J_b^p J_e^q J_r^w J_a^z \Omega_{pqwz} + g^{er} J_b^p J_e^q\prs{T_{rapq}} &\mbox{Lemma \ref{lem:Cherncurvsym}}  \\
&= g^{wq} J_b^p J_a^z \Omega_{pqwz} + g^{er} J_b^p J_e^q\prs{T_{rapq}}\\%&\mbox{simplify}  \\
&= g^{re} J_a^vJ_b^w  \Omega_{rvwe} + g^{wq} J_b^p J_a^z\prs{T_{pqwz}} + g^{er} J_b^p J_e^q\prs{T_{rapq}}. &\mbox{Lemma \ref{lem:Chernsymm}}\\
&=J_a^vJ_b^w  \cV_{vw} - g^{wq} J_b^p J_a^z\prs{T_{wzpq}} +  \gw^{qw} J_b^p\prs{T_{wapq}}. &\mbox{\eqref{eq:Tsyms}}
\end{align*}
We expand out each contracted version of $T$. First, we have that
\begin{align*}
g^{wq} J_b^p J_a^z T_{wzpq} &= g^{wq} J_b^p J_a^z\prs{ \N_p \tau_{wzq} + \N_w \tau_{qpz} + \N_q \tau_{zwp} + \N_z \tau_{pqw} } \\
&\hsp-g^{wq} J_b^p J_a^z \prs{  \tau_{psz}\tau_{wq}^s + \tau_{wsq} \tau_{zp}^s + \tau_{zsq}\tau_{pw}^s + \tau_{psw} \tau_{qz}^s } &\mbox{\eqref{eq:Tdefn}}\\
%%%%%%
&= g^{wq} J_b^p J_a^z\prs{\N_w \tau_{qpz} + \N_q \tau_{zwp}} \\
&\hsp -J_b^p J_a^z \prs{\tau_{zs}^w\tau_{pw}^s + \tau_{ps}^q \tau_{qz}^s } &\mbox{Lemma \ref{lem:tracelesstau}}\\%/simplify}\\
%%%%%%
&= g^{wq} J_b^p J_a^z\prs{\N_w \tau_{qpz} + \N_w \tau_{zqp}}\\ %&\mbox{$\prs{w,q}$ symmetric}\\
%%%%%%
&= - g^{wq} J_b^p J_a^z\prs{\N_w \tau_{pzq} }&\mbox{Proposition \ref{prop:taucyclic}}\\
%%%%%%
&= g^{wq}\prs{\N_w \tau_{baq} }. &\mbox{Corollary \ref{cor:tausymmetries} }
\end{align*}
Next we compute 
\begin{align*}
\gw^{qw} J_b^p\prs{T_{wapq}} &=\gw^{qw} J_b^p \prs{ \N_p \tau_{waq} + \N_w \tau_{qpa} + \N_q \tau_{awp} + \N_a \tau_{pqw} } \\
&\hsp - \gw^{qw} J_b^p\prs{  \tau_{psa}\tau_{wq}^s + \tau_{wsq} \tau_{ap}^s + \tau_{asq}\tau_{pw}^s + \tau_{psw} \tau_{qa}^s } &\mbox{\eqref{eq:Tdefn}}\\
%%%%%%
 &=\gw^{qw} J_b^p \prs{\N_w \tau_{qpa} + \N_q \tau_{awp} }  - \gw^{qw} J_b^p\prs{\tau_{asq}\tau_{pw}^s + \tau_{psw} \tau_{qa}^s } &\mbox{Lemma \ref{lem:tracelesstau}}\\
  &=\gw^{qw} J_b^p\N_w \prs{\tau_{qpa} - \tau_{aqp} }  - \gw^{qw} J_b^p\prs{\tau_{asq}\tau_{pw}^s + \tau_{psw} \tau_{qa}^s } \\%&\mbox{rearrange}\\
  &=- J_q^e\gw^{qw}J_p^c J_b^p\N_w \prs{\tau_{eca} - \tau_{aec} } \\
&\hsp - \gw^{qw} J_b^p\prs{-J_q^e J_s^m\tau_{ame}\tau_{cw}^n J^s_n J_p^c - J_p^c J_s^m \tau_{cmw} J_q^e J_n^s\tau_{ea}^n } &\mbox{Corollary \ref{cor:tausymmetries}}\\
 &= g^{qw}\N_w \prs{-\tau_{qba} + \tau_{aqb} }. %&\mbox{simplify}
\end{align*}
Adding these together and applying Proposition \ref{prop:taucyclic} yields
\begin{align*}
- g^{wq} J_b^p J_a^z\prs{T_{wzpq}} +  \gw^{qw} J_b^p\prs{T_{wapq}} &= g^{qw}\N_w \prs{-\tau_{baq} -\tau_{qba} + \tau_{aqb} } =- 2 g^{qw}\prs{\N_w \tau_{qab} }.
\end{align*}
The result follows.
\end{proof}
\end{lemma}
\color{black}
\begin{lemma}\label{lem:P2,0+0,2} We have that
\begin{align*}
 \cP^{2,0 + 0,2}_{ab} &= 2 \gw^{mn} \prs{\N_m \tau_{abn}}.
\end{align*}
\begin{proof}
First we manipulate $\cP$.
\begin{align*}
\cP_{ab} &= \gw^{mn} \Omega_{abmn} \\
&=\gw^{mn} \Omega_{mnab} + \gw^{mn} T_{abmn}&\mbox{Lemma \ref{lem:Chernsymm} }\\
&= J_a^v J_b^w\gw^{mn} \Omega_{mnvw} + \gw^{mn} T_{abmn}&\mbox{Lemma \ref{lem:Cherncurvsym}} \\
&= J_a^v J_b^w\gw^{mn} \prs{ \Omega_{vwmn} + T_{mnvw}  }+ \gw^{mn} T_{abmn} &\mbox{Lemma \ref{lem:Chernsymm}} \\
&=  J_a^v J_b^w \cP_{vw} +  J_a^v J_b^w\gw^{mn} T_{mnvw} + \gw^{mn} T_{abmn}. %&\mbox{simplify}
\end{align*}
For the first $T$ type term, we apply \eqref{eq:Tdefn}, the formula for $T$, and simplify.
\begin{align*}
J_a^v J_b^w \gw^{mn}  T_{mnvw} &= J_a^v J_b^w \gw^{mn} \prs{ \N_v \tau_{mnw} + \N_m \tau_{wvn} + \N_w \tau_{nmv} + \N_n \tau_{vwm} } \\
&\hsp- J_a^v J_b^w \gw^{mn} \prs{  \tau_{vsn}\tau_{mw}^s + \tau_{msw} \tau_{nv}^s + \tau_{nsw}\tau_{vm}^s + \tau_{vsm} \tau_{wn}^s } &\mbox{\eqref{eq:Tdefn}}\\
&=J_a^v J_b^w \gw^{mn} \prs{\N_m \tau_{wvn}+ \N_n \tau_{vwm} } &\mbox{Lemma \ref{lem:tracelesstau}} \\
&\hsp- J_a^v J_b^w \gw^{mn} \prs{  \tau_{vsn}\tau_{mw}^s + \tau_{msw} \tau_{nv}^s - \tau_{msw}\tau_{vn}^s - \tau_{vsn} \tau_{wm}^s }\\ %&\mbox{$(m,n)$ skew} \\
&=2 J_a^v J_b^w \gw^{mn} \prs{\N_m \tau_{wvn} } - 2 J_a^v J_b^w \gw^{mn} \prs{  \tau_{vsn}\tau_{mw}^s + \tau_{msw} \tau_{nv}^s}\\%  &\mbox{simplify} \\
&= - 2 \gw^{mn} \prs{\N_m \tau_{ban} } - 2 J_n^v J_m^w\gw^{mn} 
\prs{  \tau_{asv}\tau_{wb}^s + \tau_{wsb} \tau_{va}^s }  &\mbox{Corollary \ref{cor:tausymmetries}}\\
&=- 2 \gw^{mn} \prs{\N_m \tau_{ban} } + 2 \gw^{vw} \prs{  \tau_{asv}\tau_{wb}^s + \tau_{wsb} \tau_{va}^s }. &\mbox{\eqref{eq:identitylist}}
\end{align*}
For the second term, we have
\begin{align*}
\gw^{mn} T_{abmn} &= \gw^{mn}\prs{ \N_m \tau_{abn} + \N_a \tau_{nmb} + \N_n \tau_{bam} + \N_b \tau_{mna} } \\
&\hsp- \gw^{mn} \prs{  \tau_{msb}\tau_{an}^s + \tau_{asn} \tau_{bm}^s + \tau_{bsn}\tau_{ma}^s + \tau_{msa} \tau_{nb}^s } &\mbox{\eqref{eq:Tdefn}}\\
%%%%%%%
&= \gw^{mn}\prs{ \N_m \tau_{abn} + \N_n \tau_{bam}} &\mbox{Lemma \ref{lem:tracelesstau}}  \\
&\hsp- \gw^{wv} \prs{  \tau_{wsb}\tau_{av}^s + \tau_{asv} \tau_{bw}^s - \tau_{bsw}\tau_{va}^s - \tau_{vsa} \tau_{wb}^s } \\ %&\mbox{relabel/skewness}\\
&= 2 \gw^{mn}\prs{ \N_m \tau_{abn}} \\
&\hsp- \gw^{wv} \prs{  \tau_{wsb}\tau_{av}^s + \tau_{asv} \tau_{bw}^s} +\gw^{wv}\prs{ \tau_{bsw}\tau_{va}^s + \tau_{vsa} \tau_{wb}^s }\\%&\mbox{expand out}\\
&= 2 \gw^{mn}\prs{ \N_m \tau_{abn}} - \gw^{wv} \prs{  \tau_{wsb}\tau_{av}^s + \tau_{asv} \tau_{bw}^s}\\
&\hsp +\gw^{wv}\prs{ - \prs{ \tau_{swb} + \tau_{wbs} }\tau_{va}^s - \prs{\tau_{avs} + \tau_{sav} }\tau_{wb}^s } &\mbox{Proposition \ref{prop:taucyclic}}\\
&= 2 \gw^{mn}\prs{ \N_m \tau_{abn}} - \gw^{wv} \prs{  \tau_{wsb}\tau_{av}^s + \tau_{asv} \tau_{bw}^s}\\
&\hsp - \gw^{wv} \tau_{swb}\tau_{va}^s - \gw^{wv}\tau_{wbs}\tau_{va}^s - \gw^{wv}\tau_{avs} \tau_{wb}^s \\
&\hsp - \gw^{wv} \tau_{sav} \tau_{wb}^s \\%&\mbox{expand}\\
&= 2 \gw^{mn}\prs{ \N_m \tau_{abn}} - 2\gw^{vw} \prs{  \tau_{wsb}\tau_{va}^s + \tau_{asv} \tau_{wb}^s}. %&\mbox{simplify}
\end{align*}
Summing together these quantities, we note that the $\tau^{\ast 2}$ terms cancel out.
\end{proof}
\end{lemma}
\begin{cor}\label{cor:Rc2002} We have that
\begin{align*}
\Rc_{jk}^{2,0+0,2} = - g^{ur}\N_r \prs{ \tau_{ujk}  + \tau_{ukj}}.
\end{align*}
\begin{proof}
Starting from Corollary \ref{cor:RcOmega}, we compute the $(2,0+0,2)$ projection:
\begin{align*}
\Rc_{jk}^{2,0+0,2} &= \Omega_{jk}^{2,0+0,2} + g^{il}\prs{\N_i \tau_{klj}} &\mbox{Corollaries \ref{cor:tausymmetries}, \ref{cor:tautauid}}\\
&=  g^{il} \prs{\N_i \tau_{jlk}} + g^{il}\prs{\N_i \tau_{klj}}. &\mbox{Lemma \ref{lem:RcOmega} }
\end{align*}
The result follows.
\end{proof}
\end{cor}
\section{Variation formulas}\label{s:variationformulas}
We now examine general variation formulas of canonical objects in the almost K\"{a}hler setting: the Chern scalar curvature and the norm squared of the Chern torsion. We note it is substantially easier to consider scalar quantities, as much of the cumbersome terms are removed via the orthogonality of types.
\subsection{Torsion variation}
Here we compute the variation of the torsion tensor and the resultant variation of the norm squared of the torsion.
\begin{lemma}\label{lem:tauevolution}
We have that
\begin{align}
\begin{split}\label{eq:tauvariation}
 \dot{\tau}_{jk}^i 
&=  \tfrac{1}{4}  J_{j}^p \prs{\N_p \dot{J}_{k}^{i} - \N_k \dot{J}_p^i } -\tfrac{1}{4} J_k^p \prs{\N_p \dot{J}_j^i-\N_j \dot{J}_p^i}  + \tfrac{1}{2} \dot{J}^d_p \prs{ J_{k}^p  \tau_{jd}^i - J_j^p \tau_{kd}^i +  J_d^i \tau_{jk}^p }.
 \end{split}
\end{align}
\begin{proof}
To determine this we compute the variation of $N$, the Nijenhuis tensor, using formula \eqref{eq:Nijen}, and convert this to be in terms of the Chern connection
\begin{align*}
4 \dot{\tau}_{jk}^i &=\dot{J}_{j}^p \prs{\del_p J_{k}^{i}} - \dot{J}_{k}^p \prs{\del_p J_{j}^{i}} - \dot{J}_p^{i} \prs{\del_{j} J_{k}^p} + \dot{J}_p^{i} \prs{\del_{k} J_{j}^p} \\
&\hsp + J_{j}^p \prs{\del_p \dot{J}_{k}^{i}} - J_{k}^p \prs{\del_p \dot{J}_{j}^{i}} - J_p^{i} \prs{\del_{j} \dot{J}_{k}^p} + J_p^{i} \prs{\del_{k} \dot{J}_{j}^p} &\mbox{differentiate \eqref{eq:Nijen}}\\
&= \dot{J}_{j}^p \prs{\N_p J_{k}^{i} + \gY_{pk}^d J_d^i - \gY_{pd}^i J_k^d} - \dot{J}_{k}^p \prs{\N_p J_{j}^{i} + \gY_{pj}^d J_d^i - \gY_{pd}^i J_j^d  }\\
&\hsp - \dot{J}_p^{i} \prs{\N_{j} J_{k}^p + \gY_{jk}^d J_d^p - \gY^p_{jd} J^d_k} + \dot{J}_p^{i} \prs{\N_{k} J_{j}^p + \gY_{kj}^d J_d^p -  \gY_{kd}^p J_j^d} \\
&\hsp + J_{j}^p \prs{\N_p \dot{J}_{k}^{i} + \gY_{pk}^d \dot{J}_d^i - \gY_{pd}^i \dot{J}_k^d} - J_{k}^p\prs{\N_p \dot{J}_{j}^{i} + \gY_{pj}^d \dot{J}_d^i - \gY_{pd}^i \dot{J}_j^d  }\\
&\hsp - J_p^{i}  \prs{\N_{j} \dot{J}_{k}^p + \gY_{jk}^d \dot{J}_d^p - \gY^p_{jd} \dot{J}^d_k} + J_p^{i} \prs{\N_{k} \dot{J}_{j}^p + \gY_{kj}^d \dot{J}_d^p -  \gY_{kd}^p \dot{J}_j^d} &\mbox{$\del = \N - \gY$} \\
&= \dot{J}_{j}^p \prs{ \gY_{pk}^d J_d^i - \gY_{pd}^i J_k^d} - \dot{J}_{k}^p \prs{\gY_{pj}^d J_d^i - \gY_{pd}^i J_j^d  }  \\
&\hsp- \dot{J}_p^{i} \prs{ \gY_{jk}^d J_d^p - \gY^p_{jd} J^d_k} + \dot{J}_p^{i} \prs{ \gY_{kj}^d J_d^p -  \gY_{kd}^p J_j^d} \\
&\hsp + J_{j}^p \prs{\N_p \dot{J}_{k}^{i} + \gY_{pk}^d \dot{J}_d^i - \gY_{pd}^i \dot{J}_k^d} - J_{k}^p\prs{\N_p \dot{J}_{j}^{i} + \gY_{pj}^d \dot{J}_d^i - \gY_{pd}^i \dot{J}_j^d  }\\
&\hsp - J_p^{i}  \prs{\N_{j} \dot{J}_{k}^p + \gY_{jk}^d \dot{J}_d^p - \gY^p_{jd} \dot{J}^d_k} + J_p^{i} \prs{\N_{k} \dot{J}_{j}^p + \gY_{kj}^d \dot{J}_d^p -  \gY_{kd}^p \dot{J}_j^d}. &\mbox{\eqref{eq:Cherncharac}}
\end{align*}
The highest order terms simply group up as
\begin{align*}
4 \brk{\dot{\tau}_{jk}^i}_{J \ast \N \dot{J}}&= J_{j}^p \prs{\N_p \dot{J}_{k}^{i}} - J_k^p \prs{\N_p \dot{J}_j^i} - J_p^i \prs{\N \dot{J}_k^p} + J_p^i \prs{\N_k \dot{J}_j^p}.
\end{align*}
For the lower order terms we rearrange and label for easy combination
\begin{align*}
4 \brk{\dot{\tau}_{jk}^i}_{\dot{J} \ast \gY \ast J} &=\brk{\dot{J}_{j}^p\gY_{pk}^d J_d^i }_A \brk{- \dot{J}_{j}^p J_k^d \gY_{pd}^i }_B \brk{- \dot{J}_{k}^p \gY_{pj}^d J_d^i }_C \brk{+ \dot{J}_{k}^pJ_j^d  \gY_{pd}^i}_D\\
&\hsp  \brk{-  \gY_{jk}^d J_d^p \dot{J}_p^{i}}_E \brk{+J^d_k \gY^p_{jd} \dot{J}_p^{i} }_F \brk{+ \gY_{kj}^d J_d^p\dot{J}_p^{i}}_E \brk{  -  J_j^d \gY_{kd}^p \dot{J}_p^{i}}_G \\
&\hsp \brk{+  J_{j}^p \gY_{pk}^d \dot{J}_d^i}_G \brk{ - J_{j}^p\dot{J}_k^d\gY_{pd}^i}_D \brk{- J_{k}^p\gY_{pj}^d \dot{J}_d^i }_F\brk{+ J_{k}^p \dot{J}_j^d \gY_{pd}^i}_B \\
&\hsp \brk{- \gY_{jk}^d \dot{J}_d^p J_p^{i}}_H\brk{ +  \dot{J}^d_k \gY^p_{jd}J_p^{i}}_C \brk{+ \gY_{kj}^d \dot{J}_d^p J_p^{i}}_H \brk{-   \dot{J}_j^d \gY_{kd}^pJ_p^{i}}_A.
\end{align*}
Using this labelling scheme, for each of the following we relabel indices and combine.
\begin{align*}
A& : \dot{J}_j^p \gY_{pk}^d J_d^i - \dot{J}_j^d \gY_{kd}^p J_p^i =\dot{J}_j^d \gY_{dk}^p J_p^i - \dot{J}_j^d \gY_{kd}^p J_p^i = \dot{J}_j^d \tau_{dk}^p J_p^i  \\
B& : J_{k}^p \dot{J}_j^d \gY_{pd}^i  - \dot{J}_{j}^p J_k^d \gY_{pd}^i = \dot{J}_j^p  J_{k}^d \gY_{dp}^i  - \dot{J}_{j}^p J_k^d \gY_{pd}^i = \dot{J}_j^p  J_{k}^d \tau_{dp}^i \\
C&:   \dot{J}^d_k \gY^p_{jd}J_p^{i} - \dot{J}_{k}^p \gY_{pj}^d J_d^i = \dot{J}^p_k \gY^d_{jp}J_d^{i} - \dot{J}_{k}^p \gY_{pj}^d J_d^i = \dot{J}^p_k \tau^d_{jp}J_d^{i} \\
D&: \dot{J}_{k}^pJ_j^d  \gY_{pd}^i - J_{j}^p\dot{J}_k^d\gY_{pd}^i = \dot{J}_{k}^pJ_j^d  \gY_{pd}^i - \dot{J}_k^p J_{j}^d\gY_{dp}^i = \dot{J}_{k}^p J_j^d  \tau_{pd}^i \\
E &: \gY_{kj}^d J_d^p\dot{J}_p^{i} -  \gY_{jk}^d J_d^p \dot{J}_p^{i} =\tau_{kj}^d J_d^p\dot{J}_p^{i}\\
F&: J^d_k \gY^p_{jd} \dot{J}_p^{i}- J_{k}^p\gY_{pj}^d \dot{J}_d^i = J^d_k \gY^p_{jd} \dot{J}_p^{i}- J_{k}^d\gY_{dj}^p \dot{J}_p^i =  J^d_k \tau^p_{jd} \dot{J}_p^{i}\\
G&:   J_{j}^p \gY_{pk}^d \dot{J}_d^i-  J_j^d \gY_{kd}^p \dot{J}_p^{i} =   J_{j}^d \gY_{dk}^p \dot{J}_p^i-  J_j^d \gY_{kd}^p \dot{J}_p^{i}=   J_{j}^d \tau_{dk}^p \dot{J}_p^i\\
H&: \gY_{kj}^d \dot{J}_d^p J_p^{i}- \gY_{jk}^d \dot{J}_d^p J_p^{i} = \tau_{kj}^d \dot{J}_d^p J_p^{i}.
\end{align*}
We further combine our remaining terms
\begin{align*}
A+ B&= \dot{J}_j^d \tau_{dk}^p J_p^i +  \dot{J}_j^p  J_{k}^d \tau_{dp}^i = - \dot{J}_j^d \tau_{dp}^i J_k^p + \dot{J}_j^p  J_{k}^d \tau_{dp}^i =  2 \dot{J}_j^p  J_{k}^d \tau_{dp}^i &\mbox{Corollary \ref{cor:tausymmetries}}\\
C+ D &=  \dot{J}^p_k \tau^d_{jp}J_d^{i} +  \dot{J}_{k}^p J_j^d  \tau_{pd}^i =  -  \dot{J}^p_k \tau^i_{dp}J_j^{d} +  \dot{J}_{k}^p J_j^d  \tau_{pd}^i = 2 \dot{J}_{k}^p J_j^d  \tau_{pd}^i &\mbox{Corollary \ref{cor:tausymmetries}}\\
E + H &=\tau_{kj}^d J_d^p\dot{J}_p^{i} +   \tau_{kj}^d \dot{J}_d^p J_p^{i} = \tau_{kj}^d J_d^p\dot{J}_p^{i} -   \tau_{kj}^d J_d^p \dot{J}_p^{i} =0 &\mbox{ $\sqg{J,\dot{J}} = 0$} \\
F+G&-J^d_k \tau^p_{jd} \dot{J}_p^{i} + J_{j}^d \tau_{dk}^p \dot{J}_p^i = J^d_j \tau^p_{dk} \dot{J}_p^{i} + J_{j}^d \tau_{dk}^p \dot{J}_p^i = 2J_{j}^d \tau_{dk}^p \dot{J}_p^i. &\mbox{Corollary \ref{cor:tausymmetries}}
\end{align*}
We then combine these remaining three terms
\begin{align*}
4 \brk{\dot{\tau}_{jk}^i}_{\dot{J} \ast \gY \ast J} &= 2 \dot{J}_j^p  J_{k}^d \tau_{dp}^i + 2 \dot{J}_{k}^p J_j^d  \tau_{pd}^i + 2J_{j}^d \tau_{dk}^p \dot{J}_p^i \\
&= 2 \dot{J}_j^p  J^d_p \tau_{kd}^i +  2 \dot{J}_{k}^p J_p^d  \tau_{dj}^i - 2\dot{J}_p^i J_{d}^p \tau_{jk}^d &\mbox{Corollary \ref{cor:tausymmetries}} \\
&=  - 2 J_j^p  \dot{J}^d_p \tau_{kd}^i -  2 J_{k}^p \dot{J}_p^d  \tau_{dj}^i + 2J_p^i \dot{J}_{d}^p \tau_{jk}^d  &\mbox{ $\sqg{J,\dot{J}} = 0$} \\
&=2 \dot{J}^d_p \prs{ J_{k}^p  \tau_{jd}^i - J_j^p \tau_{kd}^i  +  J_d^i \tau_{jk}^p },
\end{align*}
combining yields the result.
\end{proof}
\end{lemma}
While the variation of $\tau$ is not so clean, the formula for the variation of $\brs{\tau}^2$ is comparatively simple. We note the appearance below of precisely the form of the evolution of $J$ along symplectic curvature flow featured in Proposition \ref{prop:Jdot}, which hints to the naturality of such evolution.
\begin{lemma}\label{lem:normtauvar}We have that
\begin{align}
\begin{split}\label{eq:|tau|var}
\delta\brs{ \tau }^2 &=\ip{\dot{g} , \cB } + 2 g^{ap} \gw^{bw}\tau_{bae}\prs{\N_p \dot{J}_w^e}  \\
&= \ip{\dot{g} , \cB }-  \tfrac{1}{2} g_{ey}g^{uw} \prs{4 \gw^{pv} \prs{ \N_p\tau_{vu}^y}}\dot{J}_w^e  +2 \N_p \brk{g^{ap}  \gw^{bw}\tau_{bae}  \dot{J}_w^e} .
\end{split}
\end{align}
\begin{proof} Using Corollary \ref{cor:Btraces}, we compute
\begin{align*}
\tfrac{1}{2} \delta \brs{ \tau }^2 &= \delta \brk{g^{ij} \tau_{ie}^w \tau_{jw}^e} = - \dot{g}_{cd} g^{ci} g^{dj} \tau_{ie}^w \tau_{jw}^e +2 g^{ij} \tau_{ie}^w \dot{\tau}_{jw}^e.
\end{align*}
We examine the second term:
\begin{align*}
2g^{ij} \tau_{ie}^w \dot{\tau}_{jw}^e 
&=\tfrac{1}{2}  g^{ij} \tau_{ie}^w J_{j}^p \prs{\N_p \dot{J}_{w}^{e} - \N_w \dot{J}_p^e } -\tfrac{1}{2} g^{ij} \tau_{ie}^w J_w^p \prs{\N_p \dot{J}_j^e-\N_j \dot{J}_p^e}&\mbox{(Row 1)}\\
&\hsp+ g^{ij} \tau_{ie}^w \dot{J}^d_p \prs{ J_{w}^p  \tau_{jd}^e - J_j^p \tau_{wd}^e +  J_d^e \tau_{jw}^p }. &\mbox{(Row 2)}
\end{align*}
For (Row 1) we first simplify the terms hitting the quantities in parantheses.
\begin{equation}\label{eq:randoids}
g^{ij} \tau_{ie}^w J_j^p = \gw^{pi}\tau_{ie}^w \qquad g^{ij} \tau_{ie}^w J_w^p = - g^{ij} \tau_{we}^p J_i^w =- \gw^{wj} \tau_{we}^p.
\end{equation}
Thus, for (Row 1):
\begin{align*}
\text{(Row 1)} &= \tfrac{1}{2} g^{ij} \tau_{ie}^w J_{j}^p \prs{\N_p \dot{J}_{w}^{e} - \N_w \dot{J}_p^e } -\tfrac{1}{2} g^{ij} \tau_{ie}^w J_w^p \prs{\N_p \dot{J}_j^e-\N_j \dot{J}_p^e}\\
&= \tfrac{1}{2}\gw^{pi}\tau_{ie}^w \prs{\N_p \dot{J}_{w}^{e} - \N_w \dot{J}_p^e } + \tfrac{1}{2} \gw^{wj} \tau_{we}^p \prs{\N_p \dot{J}_j^e-\N_j \dot{J}_p^e} &\mbox{\eqref{eq:randoids}}\\
&= \tfrac{1}{2} \gw^{pi}\tau_{ie}^w \prs{\N_p \dot{J}_{w}^{e} - \N_w \dot{J}_p^e } + \tfrac{1}{2}\gw^{ip} \tau_{ie}^w \prs{\N_w \dot{J}_p^e-\N_p \dot{J}_w^e}\\ %&\mbox{relabel}\\
&= \tfrac{1}{2}\gw^{pi}\tau_{ie}^w \prs{\N_p \dot{J}_{w}^{e} - \N_w \dot{J}_p^e } + \tfrac{1}{2}\gw^{pi} \tau_{ie}^w \prs{\N_p \dot{J}_w^e-\N_w \dot{J}_p^e}\\ %&\mbox{reorganize, $\gw$ skew}\\
&= \gw^{pi}\tau_{ie}^w \prs{\N_p \dot{J}_{w}^{e} - \N_w \dot{J}_p^e }. %&\mbox{simplify}
\end{align*}
We manipulate the first term of this quantity further:
\begin{align*}
\gw^{pi} \tau_{ie}^w \prs{\N_p \dot{J}_w^e} &=\gw^{pi}g^{wc} \tau_{iec} \prs{\N_p \dot{J}_w^e} \\
&=\gw^{pi}g^{wc}\prs{-J_i^a J_c^b  \tau_{aeb}} \prs{\N_p \dot{J}_w^e}  &\mbox{Corollary \ref{cor:tausymmetries}} \\
&=- g^{ap} \gw^{bw}\tau_{aeb} \prs{\N_p \dot{J}_w^e} \\
&= g^{ap} \gw^{bw}\prs{\tau_{bae} + \tau_{eba} }\prs{\N_p \dot{J}_w^e}  &\mbox{Proposition \ref{prop:taucyclic}} \\
&= g^{ap} \gw^{bw}\tau_{bae}\prs{\N_p \dot{J}_w^e} + \gw^{bw}\tau_{eb}^p\prs{\N_p \dot{J}_w^e} \\
%&= g^{ap} \gw^{bw}\tau_{bae}\prs{\N_p \dot{J}_w^e} + \gw^{ip}\tau_{ei}^w\prs{\N_w \dot{J}_p^e}  \\%&\mbox{relabel: $b\mapsto i$, $w \leftrightarrow p$}\\
&= g^{ap} \gw^{bw}\tau_{bae}\prs{\N_p \dot{J}_w^e} + \gw^{pi}\tau_{ie}^w\prs{\N_w \dot{J}_p^e}.
\end{align*}
Therefore we have that
\begin{equation*}
\text{(Row 1)} = g^{ap} \gw^{bw}\tau_{bae}\prs{\N_p \dot{J}_w^e}.
\end{equation*}
For (Row  2) we compare the first and last term. Manipulating the last, we get that
\begin{align*}
g^{ij} \tau_{ie}^w \dot{J}_p^d J_d^e \tau_{jw}^p &=  g^{ij} \tau_{iw}^p \dot{J}_p^d J_d^e \tau_{je}^w \\ %&\mbox{reorganize, $(i,j)$ symmetry} \\
&=  -g^{ij} \tau_{iw}^p J_p^d \dot{J}_d^e \tau_{je}^w &\mbox{$\sqg{J,\dot{J}}=0$} \\
&=  -g^{ij} \tau_{ie}^w J_w^p \dot{J}_p^d \tau_{jd}^e.% &\mbox{relabel: $ e \mapsto d$, $d \mapsto p$, $p \mapsto w$, $w \mapsto e$}
\end{align*}
Therefore,
\begin{equation*}
g^{ij} \tau_{ie}^w \dot{\tau}_{jw}^e 
=- \tfrac{1}{2}g^{ij} \tau_{ie}^w \dot{J}^d_p J_j^p \tau_{wd}^e = \tfrac{1}{2}\gw^{ip} \tau_{ie}^w  \tau_{wd}^e \dot{J}^d_p.
\end{equation*}
We argue that additionally the last term vanishes. To see this, we note that $\gw^{-1}$ is type $(1,1)$. Comparatively, the remainder of this term is type $(2,0 + 0,2)$. To see this observe that
\begin{align*}
J_i^a J_p^b\tau_{ae}^w \dot{J}^d_b \tau_{wd}^e &= -J_i^a \dot{J}_p^b\tau_{ae}^w J^d_b \tau_{wd}^e  &\mbox{$\sqg{J,\dot{J}} = 0$}\\
&= J_i^a \dot{J}_p^b\tau_{ae}^w \prs{J^e_d \tau_{wb}^d}  &\mbox{Corollary \ref{cor:tausymmetries}}\\
&=  \dot{J}_p^b \prs{J_i^a \tau_{ae}^w J^e_d} \tau_{wb}^d \\%&\mbox{regroup}\\
&= - \dot{J}_p^b \tau_{id}^w \tau_{wb}^d  &\mbox{Corollary \ref{cor:tausymmetries}}\\
&= - \dot{J}_p^d \tau_{ie}^w \tau_{wd}^e. % &\mbox{relabel}
\end{align*}
Thus $\text{(Row 2)} = 0$, which yields the first line of the \eqref{eq:|tau|var}. The second line follows by shifting the Chern derivative at the expense of a divergence term, then applying Corollary \ref{cor:tausymmetries}.
\end{proof}
\end{lemma}
\subsection{Scalar curvature variation}
In this section we begin computing variations of curvature objects.
\noindent Recall the evolution for the Riemannian scalar curvature is
\begin{equation}\label{eq:R}
\dot{\cR} =- g^{ip} g^{qj}\dot{g}_{pq} \Rc_{ij} - \lap_D \prs{g^{ij}\dot{g}_{ij}} + g^{ip} g^{jq} D_j  D_i \dot{g}_{pq},
\end{equation}
where here $\lap_D$ is the Levi--Civita rough Laplacian, given in coordinates by $\lap_D \prs{\cdot} := g^{ij} D_i D_j \prs{\cdot}$. We convert the second and third terms of \eqref{eq:R} to be in terms of the Chern connection.
\begin{prop}\label{prop:RChernconn} We have that
\begin{align*}
\dot{\cR} %&=- g^{ip} g^{qj}\dot{g}_{pq} \Rc_{ij}-  \lap \prs{g^{ij} \dot{g}_{ij}}  +  g^{ip} g^{jq}\prs{ \N_j \N_i \dot{g}_{pq}}- g^{dw}g^{qj}\N_j \brk{\dot{g}_{pd} \tau_{qw}^p } \\
&=- \ip{\dot{g}, \Rc} - \lap \prs{\tr_g \dot{g}} +   g^{ip} g^{jq}\prs{ \N_j \N_i \dot{g}_{pq}}- g^{dw}g^{qj}\N_j \brk{\dot{g}_{pd} \tau_{qw}^p }.
\end{align*}
\begin{proof} We convert each piece of \eqref{eq:R}. First,
\begin{align*}
- \lap_D \prs{ g^{ij} \dot{g}_{ij}} &= - g^{uv}\prs{ D_u D_v g^{ij} \dot{g}_{ij}}\\
&= - g^{uv}\prs{ D_u \N_v g^{ij} \dot{g}_{ij}}\\ %&\mbox{convert inner Chern (free)}\\
&= - g^{uv}\prs{ \N_u \N_v g^{ij} \dot{g}_{ij}} + g^{uv}  \Theta_{uv}^d \prs{\N_d g^{ij} \dot{g}_{ij}} \\%&\mbox{convert outer Chern: $D=  \N + \Theta$}  \\
&= -  \lap \prs{g^{ij} \dot{g}_{ij}}.% &\mbox{simplify, note $\tr_g \Theta^k = 0$}
\end{align*}
Next we have
\begin{align*}
g^{ip} g^{jq} D_j D_i \dot{g}_{pq} &= g^{ip} g^{jq} D_j \prs{\N_i \dot{g}_{pq} - \Theta_{ip}^d \dot{g}_{dq} - \Theta_{iq}^d \dot{g}_{pd} } &\mbox{\eqref{eq:Cherndef}}\\
&= g^{ip} g^{jq} \N_j \prs{\N_i \dot{g}_{pq} - \Theta_{ip}^d \dot{g}_{dq} - \Theta_{iq}^d \dot{g}_{pd} } \\
&\hsp- g^{ip} g^{jq} \Theta_{jq}^e \prs{\N_i \dot{g}_{pe} - \Theta_{ip}^d \dot{g}_{de} - \Theta_{ie}^d \dot{g}_{pd} } &\mbox{\eqref{eq:Cherndef}}\\
&= g^{ip} g^{jq} \N_j \prs{\N_i \dot{g}_{pq} - \Theta_{iq}^d \dot{g}_{pd} } &\mbox{Lemma \ref{lem:tracelesstau}} \\
&= g^{ip} g^{jq}\prs{ \N_j \N_i \dot{g}_{pq}} -  \N_j  \prs{\Theta_{iq}^d \dot{g}_{pd} }\\
&= g^{ip} g^{jq}\prs{ \N_j \N_i \dot{g}_{pq}} - g^{jq}g^{dw} \N_j  \brk{\tau_{qw}^p \dot{g}_{pd} }.&\mbox{\eqref{eq:Cherndef}}
\end{align*}
The result follows.
\end{proof}
\end{prop}
\section{Evolutions along symplectic curvature flow}\label{s:SCFevs}
Here we provide evolutions along symplectic curvature flow of key quantities. We simplify the evolution equation of $J$ (Proposition \ref{prop:Jdot}), and derive evolution equations for the norm of Chern torsion and scalar curvature (Theorem \ref{thm:Cherntors}, \ref{thm:Chernscalcurv}).
\subsection{Almost complex structure evolution}\label{eq:eqns}
Recall that the $J$ evolution along symplectic curvature flow is given in (\cite{ST}, in the proof of Lemma 9.2) by
\begin{align}
\begin{split}\label{eq:Jevorig}
\prs{\tfrac{\del J}{\del t}}_{\mathbf{i}}^{\mathbf{k}} &= -\prs{\cP_{\mathbf{i}w}^{2,0+0,2} - 2 J_{\mathbf{i}}^e \Rc_{ew}^{2,0+0,2} } g^{w\mathbf{k}}.
\end{split}
\end{align}
\begin{proof}[Proof of Proposition \ref{prop:Jdot}]
We compute
\begin{align*}
2 J_j^e \Rc^{2,0+0,2}_{el} -\cP^{2,0+0,2}_{jl}
&=- 2 J_j^e g^{ur} \prs{\N_u \tau_{rel} + \N_u \tau_{rle}} + 2 \gw^{rk} \N_r \tau_{jlk}. &\mbox{Corollary \ref{cor:Rc2002},  Lemma \ref{lem:P2,0+0,2}} \\
&=- 2 \gw^{yu} \prs{\N_u \tau_{yjl} + \N_u \tau_{ylj}} + 2 \gw^{rk} \N_r \tau_{jlk}. &\mbox{Corollary \ref{cor:tausymmetries}} \\
&= -2 \gw^{mr} \prs{\N_r \tau_{mjl} + \N_r \tau_{mlj}} + 2 \gw^{rm} \N_r \tau_{jlm}.\\% &\mbox{relabel} \\
&= -2 \gw^{mr} \prs{\N_r \tau_{mjl}} + 2 \gw^{mr}\N_r\prs{ \tau_{lmj} +  \tau_{jlm}}.\\%&\mbox{combine} \\
&= -4 \gw^{mr} \prs{\N_r \tau_{mjl}}. &\mbox{Proposition \ref{prop:taucyclic}}
\end{align*}
Raising by $g^{-1}$ yields the desired result.
\end{proof}
\noindent Thus, symplectic curvature flow can be written alternatively using the following evolution equations
\begin{equation}\label{eq:SCF}
\begin{cases}
\prs{\tfrac{\del \gw}{\del t}}_{\mathbf{ij}} &= - \cP_{\mathbf{ij}} \\
\prs{\tfrac{ \del J}{\del t}}_{\mathbf{i}}^{\mathbf{k}} &=4 \gw^{re} \prs{\N_r \tau_{e \mathbf{i}}^{\mathbf{k}}} \\
\prs{\tfrac{\del g}{\del t}}_{\mathbf{ij}} &= - J_{\mathbf{i}}^s \cP_{\mathbf{j}s} + 4g^{ra}\prs{\N_r \tau_{a\mathbf{ij}}}.
\end{cases}
\end{equation}
\begin{rmk}
Note $\mathcal{H}^1$ mentioned in (\cite{ST}, (6.4)) vanishes identically. Compared to \cite{ST} there is a conventional factor of $2$ difference from the definition of $N$ (cf. \cite{ST} (2.1)).
\end{rmk}
\subsection{Chern norm squared evolution}
We now approach the evolution of the Chern torsion norm squared. We begin with a lemma which gives an alternate form for a key term in the proof of Theorem \ref{thm:Cherntors}.
\begin{lemma}\label{lem:laplacetauid}
We have that
\begin{align*}
\lap \brs{\tau}^2 &=-2  \ip{\Rc^{1,1}, \cB}+ \brs{\cB}^2  +  8 g^{ap}g^{bq}g^{cs}\prs{\N_a \N_r \tau_{rbc}} \tau_{pqs} + 2 \brs{\N \tau}^2 \\
&\hsp +  4g^{ap}g^{bq} g^{dr}g^{es} \Omega_{rpqs}\prs{ \tau_{edc} \tau_{abc} +  \tau_{cde} \tau_{acb} + \tau_{ced}\tau_{acb}}.
\end{align*}
\begin{proof} Since we have removed all $J$, $\omega$ from this context, all contractions are in terms of $g$, so we may simplify our notation: matching indices denotes contraction by the metric. Noting that \[ \lap \brs{\tau}^2 = 2 \prs{\lap \tau_{abc}} \tau_{abc} + 2 \brs{\N \tau}^2, \] we manipulate the first term.
\begin{align*}
2 \prs{\lap \tau_{abc}} \tau_{abc} &= 2 \prs{\N_d \N_d \tau_{abc}} \tau_{abc} \\
&= -2 \N_d \prs{\N_a \tau_{bdc} + \N_b \tau_{dac} } \tau_{abc} \\
&\hsp + 2 \N_d \prs{\Omega_{abdc} + \Omega_{dabc} + \Omega_{bdac}} \tau_{abc} \\
&\hsp + 2 \N_d \prs{ \tau_{asc} \tau_{bd}^s + \tau_{bsc} \tau_{da}^s + \tau_{dsc} \tau_{ab}^s } \tau_{abc} &\mbox{Theorem \ref{thm:CBianchi}}\\
&= - 4 \prs{\N_d \N_a \tau_{bdc} } \tau_{abc} + 2 \prs{\N_d \Omega_{abdc}} \tau_{abc}\\
&=- 4 \prs{\N_a \N_d \tau_{bdc} } \tau_{abc}-4 \prs{\brk{\N_d, \N_a} \tau_{bdc} } \tau_{abc}\\
&\hsp + 2\prs{ \N_d \Omega_{abdc}} \tau_{abc}. &\mbox{Lemma \ref{lem:commutator}}
\end{align*}
The first term is in the desired form. We now manipulate the last curvature quantity,
\begin{align*}
2 \prs{\N_d \Omega_{abdc}} \tau_{abc} &= -2 \prs{\N_a \Omega_{bddc} + \N_b \Omega_{dadc}} \tau_{abc} \\
&\hsp +2 \prs{\Omega_{amdc} \tau_{bdm} + \Omega_{bmdc} \tau_{dam} + \Omega_{dmdc} \tau_{abm}} \tau_{abc} &\mbox{Theorem \ref{thm:CBianchi}} \\
&= - 4 \prs{\N_a \Omega_{bddc}} \tau_{abc} +  4 \Omega_{amdc} \tau_{bdm}\tau_{abc}  + 2 \Omega_{dmdc} \tau_{abm}\tau_{abc}\\
&= - 4 \prs{\N_a \cV_{bc}^{2,0+0,2}} \tau_{abc} +4 \Omega_{amdc} \tau_{bdm}\tau_{abc}  - \tfrac{1}{2} \ip{ \cV^{1,1}, \cB^1}&\mbox{Lemma \ref{lem:Bformulas} } \\
&= 4 \prs{\N_a \N_r \tau_{rbc}} \tau_{abc} +4 \Omega_{amdc} \tau_{bdm}\tau_{abc}  - \tfrac{1}{2} \ip{ \cV^{1,1}, \cB^1}. &\mbox{Lemma \ref{lem:RcOmega} }
\end{align*}
We now expand out the commutator.
\begin{align*}
-4 \prs{\brk{\N_d, \N_a} \tau_{bdc} } \tau_{abc}  &= 4 \prs{\Omega_{dabe} \tau_{edc} +\Omega_{dade} \tau_{bec} + \Omega_{dace} \tau_{bde} + \tau_{dae} \N_e \tau_{bdc} }\tau_{abc} &\mbox{Lemma \ref{lem:commutator}} \\
 &= 4 \prs{\Omega_{dabe} \tau_{edc} -\cV_{ae}^{1,1} \tau_{bec} + \Omega_{dace} \tau_{bde} } \tau_{abc} &\mbox{Corollary \ref{cor:tautauid}}\\
 &= \ip{\cV^{1,1}, \cB^2}+ 4 \Omega_{dabe}\prs{ \tau_{edc} \tau_{abc} +  \tau_{cde} \tau_{acb} }. &\mbox{Lemma \ref{lem:Bformulas}}
\end{align*}
Adding these all together combining $\cB$ type terms using \eqref{eq:cBforms} and simplifying yields
\begin{align*}
2 \prs{\lap \tau_{abc}} \tau_{abc}
%&=\ip{\cV^{1,1}, \cB^2 - \tfrac{1}{2} \cB^1}+  8 \prs{\N_a \N_r \tau_{rbc}} \tau_{abc}\\
%&\hsp + 4 \Omega_{dabe}\prs{ \tau_{edc} \tau_{abc} +  \tau_{cde} \tau_{acb} } +4 \Omega_{amdc} \tau_{bdm}\tau_{abc}\\
%%%%
&=-2 \ip{\cV^{1,1}, \cB}+  8 \prs{\N_a \N_r \tau_{rbc}} \tau_{abc} + 4 \Omega_{dabe}\prs{ \tau_{edc} \tau_{abc} +  \tau_{cde} \tau_{acb} + \tau_{ced}\tau_{acb}}.
\end{align*}
Now let's take a look at $\cV^{1,1}$. Projecting Corollary \ref{cor:RcOmega} onto the $(1,1)$ part gives
\begin{equation}\label{eq:VBstuff}
\Rc_{jk}^{1,1} %&= \cV_{jk}^{1,1}  -\tau_{dlj} \tau_{kdl}  + \tau_{djc}\tau_{kdc} \\
= \cV_{jk}^{1,1}  -\tau_{dlj} \tau_{kdl}  - \tau_{jdc}\tau_{kdc}.
\end{equation}
The last term is a multiple of $\cB^2$. Now we take a moment to observe that
\begin{align*}
\tau_{dlj} \tau_{kdl} &= -\tau_{dlj} \prs{ \tau_{dlk} + \tau_{lkd} }  &\mbox{Lemma \ref{prop:taucyclic}}\\
&= -\tau_{dlj} \tau_{dlk} -\tau_{dlj}\tau_{lkd}\\
&=   - \tfrac{1}{2} \tau_{dlj} \tau_{dlk} &\mbox{(identify with start term)}\\
&= - \tfrac{1}{8} \cB^1_{jk}. &\mbox{Lemma \ref{lem:Bformulas}}
\end{align*}
Then updating \eqref{eq:VBstuff} produces the following equality
\begin{align*}
\Rc_{jk}^{1,1} &= \cV_{jk}^{1,1}  + \tfrac{1}{8} \cB^1_{jk}   -  \tfrac{1}{4} \cB^2_{jk}= \cV^{1,1}_{jk} + \tfrac{1}{2} \cB_{jk}.
\end{align*}
Combining yields the result.
\end{proof}
\end{lemma}
\begin{proof}[Proof of Theorem \ref{thm:Cherntors}] We have that
\begin{align*}
\tfrac{\del}{\del t}\brs{ \tau }^2 &=2 g^{ap} \gw^{bw}\tau_{bae}\prs{\N_p \prs{\tfrac{\del J}{\del t}}_w^e} + \ip{\tfrac{\del g}{\del t}, \cB} &\mbox{Lemma \ref{lem:normtauvar}}\\
&=8 g^{ap} \gw^{bw} \gw^{uv}\tau_{bae}\prs{\N_p \N_u \tau_{vw}^e} -2 \ip{ \Rc, \cB} + 2 \brs{\cB}^2 &\mbox{\eqref{eq:SCF}}\\
&=- 8 g^{ap} J_w^m \gw^{bw} J_v^n\gw^{uv}\tau_{bae}\prs{\N_p \N_u \tau_{nm}^e} -2 \ip{ \Rc, \cB} + 2 \brs{\cB}^2 &\mbox{Corollary \ref{cor:tausymmetries}}\\
&= 8  \tau_{pbe}\prs{\N_p \N_u \tau_{ube}} -2 \ip{ \Rc, \cB} + 2 \brs{\cB}^2\\
&= \lap \brs{\tau}^2 + \brs{\cB}^2 - 2 \brs{\N \tau}^2 -  4g^{ap}g^{bq} g^{dr}g^{es} \Omega_{rpqs}\prs{ \tau_{edc} \tau_{abc} +  \tau_{cde} \tau_{acb} + \tau_{ced}\tau_{acb}}, &\mbox{Lemma \ref{lem:laplacetauid}}
\end{align*}
which yields the first result of Theorem \ref{thm:Cherntors}. We analyze the last term using reduced notation again (so contractions by $g$ will be denoted by matching indices).
\begin{equation}\label{eq:Omegatautauform}
\prs{\star} = -  4 \Omega_{dabe}\prs{ \tau_{edc} \tau_{abc} +  \tau_{cde} \tau_{acb} + \tau_{ced}\tau_{acb}}.
\end{equation}
Let's consider
\begin{align*}
\Omega_{dabe} \tau_{cde}\tau_{acb} &= - \prs{\Omega_{abde} + \Omega_{bdae} }\tau_{cde}\tau_{acb} \\
&\hsp- \prs{ \tau_{ase}\tau_{bds} +  \tau_{bse}\tau_{das} +  \tau_{dse}\tau_{abs} }\tau_{cde}\tau_{acb} &\mbox{Theorem \ref{thm:CBianchi}}\\
&= - \Omega_{bdae} \tau_{cde}\tau_{acb}-  \tau_{dse}\tau_{abs} \tau_{cde}\tau_{acb}. \\
&= Q^1 + Q^2.
\end{align*}
We manipulate the first term
\begin{align*}
Q^1 = - \Omega_{bdae} \tau_{cde} \tau_{acb} &= - \Omega_{daeb} \tau_{cab} \tau_{ecd} \\
&= \Omega_{dabe} \tau_{cab} \tau_{ecd} &\mbox{Lemma \ref{lem:Cherncurvsym}}\\
&= \Omega_{dabe} \tau_{acb} \tau_{ced}. &\mbox{Corollary \ref{cor:tausymmetries}}
\end{align*}
Next we have that
\begin{align*}
Q^2 = -  \prs{ \tau_{dse}\tau_{cde}}\prs{\tau_{abs}\tau_{acb}} &=   \prs{ \tau_{sde}\tau_{cde}}\prs{ -\prs{ \tau_{sab} + \tau_{bsa} }\tau_{acb}} &\mbox{Proposition \ref{prop:taucyclic} }\\
%&= \tfrac{1}{4} \cB^2_{sc}\prs{ -\tau_{sab}\tau_{acb} - \tau_{bsa} \tau_{acb}} &\mbox{Lemma \ref{lem:Bformulas}}\\
%&= \tfrac{1}{4} \cB^2_{sc}\prs{ \tau_{sab}\tau_{cab} - \tau_{sba} \tau_{cab}}\\
&= \tfrac{1}{4} \cB^2_{sc}\prs{ \tfrac{1}{4} \cB^2_{sc} + \tfrac{1}{2} \cB_{sc}} &\mbox{Lemma \ref{lem:Bformulas}} \\
&=  \tfrac{1}{4} \prs{ \tfrac{1}{8} \ip{\cB^2, \cB^1}}.
\end{align*}
Therefore updating \eqref{eq:Omegatautauform} we have
\begin{align*}
\prs{\star} %&=  -  4 \Omega_{dabe}\prs{ \tau_{edc} \tau_{abc} +  2 \tau_{ced}\tau_{acb}}  - \tfrac{1}{4} \brs{\cB^2}^2 - \tfrac{1}{2} \ip{\cB^2, \cB} \\
%&=  -  4 \Omega_{dabe}\prs{ \tau_{edc} \tau_{abc} +  2 \tau_{ced}\tau_{acb}}  - \tfrac{1}{4} \brs{\cB^2}^2 - \tfrac{1}{2} \ip{\cB^2, \tfrac{1}{4} \cB^1 - \tfrac{1}{2} \cB^2} \\
%&=  -  4 \Omega_{dabe}\prs{ \tau_{edc} \tau_{abc} +  2 \tau_{ced}\tau_{acb}}  - \tfrac{1}{4} \brs{\cB^2}^2 +   \tfrac{1}{4} \brs{\cB^2}^2 - \tfrac{1}{8} \ip{\cB^2, \cB^1} \\
&=  -  4 \Omega_{dabe}\prs{ \tau_{edc} \tau_{abc} +  2 \tau_{ced}\tau_{acb}} - \tfrac{1}{8} \ip{\cB^2, \cB^1}.
\end{align*}
We recall that, using \eqref{eq:cBforms},
\begin{equation*}\label{eq:Bnormcomp}
\brs{\cB}^2 = \brs{\tfrac{1}{4} \cB^1 - \tfrac{1}{2} \cB^2}^2 = \tfrac{1}{16} \brs{\cB^1}^2 + \tfrac{1}{4} \brs{\cB^2}^2 - \tfrac{1}{4} \ip{\cB^1, \cB^2}.
\end{equation*}
Thus for $\dim_{\mathbb{C}} M = 2$, via Lemma \ref{lem:Dai14.10}, we have
\begin{equation}\label{eq:Btermstau4}
\brs{\cB^1}^2= 8 \brs{\tau}^4, \quad \brs{\cB^2}^2 = 4 \brs{\tau}^4, \quad \ip{\cB^1, \cB^2} = 4 \brs{\tau}^4, \text{ so } \brs{\cB}^2 = \tfrac{1}{2} \brs{\tau}^4 + \brs{\tau}^4- \brs{\tau}^4 = \tfrac{1}{2} \brs{\tau}^4.
\end{equation}
Incorporating these identities in, the second result of Theorem \ref{thm:Cherntors} for $\dim_{\mathbb{C}} M =2$ follows.
%
%\begin{align*}
%\tfrac{\del}{\del t}\brs{ \tau }^2 
%%&= \lap \brs{\tau}^2 + \brs{\cB}^2 - 2 \brs{\N \tau}^2 -  4 \Omega_{dabe}\prs{ \tau_{edc} \tau_{abc} +  \tau_{cde} \tau_{acb} + \tau_{ced}\tau_{acb}}, &\mbox{Lemma \ref{lem:laplacetauid}} \\
%&= \lap \brs{\tau}^2  - 2 \brs{\N \tau}^2 -  4 \Omega_{dabe}\prs{ \tau_{edc} \tau_{abc} +  2 \tau_{ced}\tau_{acb}} + \brs{\cB}^2 - \tfrac{1}{8} \ip{\cB^2, \cB^1}, &\\
%&= \lap \brs{\tau}^2  - 2 \brs{\N \tau}^2 -  4 \Omega_{dabe}\prs{ \tau_{edc} \tau_{abc} -  2 \tau_{ced}\tau_{cab}}. &\mbox{$\dim_{\mathbb{C}} M =2$}
%\end{align*}
%
%The second result follows.
\end{proof}
\begin{proof}[Proof of Corollary \ref{cor:Cherncurvcont}]
Referring to \eqref{eq:SCF} for symplectic curvature flow in terms of Chern connection quantities, it is clear that one needs appropriate control of $\brs{\N \tau}^2$, which follows naturally from a combination of the assumed control of $\Omega$ and the subsequent control of $\brs{\tau}^2$ combined with Theorem 7.10 of \cite{ST}.
\end{proof}
\subsection{Scalar curvature evolution}\label{ss:chernscalcurv}
We now give the evolution of the corresponding Chern scalar curvature.
\begin{proof}[Proof of Theorem \ref{thm:Chernscalcurv}]
Using Corollary \ref{cor:scalcurv} we divide the variation into two pieces.
\begin{align*}
\tfrac{1}{2} \tfrac{\del \varrho}{\del t} &= \tfrac{\del \cR}{\del t} + \tfrac{\del}{\del t} \brs{\tau}^2.
\end{align*}
First we have that, considering the variation of $g$ as a modified Ricci flow,
\begin{align*}
\tfrac{\del \cR}{\del t} &= \lap \cR+2  \lap \brs{\tau}^2 + 2 \brs{\Rc}^2 - 2 \ip{\cB, \Rc}  + 2  g^{ip} g^{jq}\prs{ \N_j \N_i \cB_{pq}} &\mbox{Proposition \ref{prop:RChernconn} } \\
 &= \lap \cR+2  \lap \brs{\tau}^2 +\prs{ 2 \brs{\Rc}^2 - 4 \ip{\cB, \Rc}+2 \brs{\cB}^2}\\
 &\hsp  + 2 \ip{\cB, \Rc} - 2 \brs{\cB}^2 + 2  g^{ip} g^{jq}\prs{ \N_j \N_i \cB_{pq}} &\mbox{(insert/remove $\brs{\cB}^2$, $\ip{\cB,\Rc}$)}\\
 &= \lap \cR+2  \lap \brs{\tau}^2 +2 \brs{\Rc- \cB}^2 + 2 \ip{\cB, \Rc} - 2 \brs{\cB}^2\\
 &\hsp + 2  g^{ip} g^{jq}\prs{ \N_j \N_i \cB_{pq}}.
\end{align*}
Combining with  Theorem \ref{thm:Cherntors} we note the convenient combination of terms
%Additionally we have that
%%
%\begin{align*}
%\tfrac{\del}{\del t} \brs{\tau}^2 &= \lap \brs{\tau}^2 + \brs{\cB}^2 - 2 \brs{\N \tau}^2 -  4 \Omega_{dabe}\prs{ \tau_{edc} \tau_{abc} +  \tau_{cde} \tau_{acb} + \tau_{ced}\tau_{acb}}.
%\end{align*}
%
%Summing these up we have
%
\begin{align*}
\tfrac{1}{2} \tfrac{\del \varrho}{\del t} &= \tfrac{\del}{\del t}\prs{\cR + \brs{\tau}^2 } \\
&=  \lap \cR +2 \brs{\Rc- \cB}^2 +2  \lap \brs{\tau}^2+ 2  g^{ip} g^{jq}\prs{ \N_j \N_i \cB_{pq}}\\
&\hsp +  \lap \brs{\tau}^2 - 2 \brs{\N \tau}^2 + 2 \ip{\cB, \Rc} - \brs{\cB}^2  \\
&\hsp -   4g^{ap}g^{bq} g^{dr}g^{es} \Omega_{rpqs}\prs{ \tau_{edc} \tau_{abc} +  \tau_{cde} \tau_{acb} + \tau_{ced}\tau_{acb}} \\
&=\tfrac{1}{2}  \lap \varrho +2 \brs{\Rc- \cB}^2 +\lap \brs{\tau}^2 + 2  \prs{ \N_j \N_i \cB_{ij}} +  8 \prs{\N_a \N_r \tau_{rbc}} \tau_{abc}. &\mbox{Lemma \ref{lem:laplacetauid}}
\end{align*}
The result follows.
\end{proof}
\section{Rigidity result}\label{s:rigidity}
We now give an improvement of the following classification of static points of \cite{ST}.
\begin{customcor}{9.5 of \cite{ST}}\label{cor:2.5ofST}
A compact static structure $\prs{M^4,\gw,J}$ is K\"{a}hler--Einstein.
\end{customcor}
\noindent By \emph{static} the authors mean a solution to symplectic curvature flow for which there exists a $\la \in \mathbb{R}$ such that
\begin{align}\label{eq:staticdef}
\begin{cases}
\tfrac{\del \gw_t}{\del t} = - \la \gw_t \qquad &\gw_0 = \left. \gw_t \right|_{t=0}, \\
\tfrac{\del J_t}{\del t} = 0 \qquad &J_0 = \left. J_t \right|_{t=0}. \\ 
\end{cases}
\end{align}
The first condition arises for solutions where one rescales by the metric, while the second is a natural assumption since one cannot scale almost complex structures. Static structures are expected smooth limit points of symplectic curvature flow. The method for proving Corollary 2.5 of \cite{ST} relies on Theorem 2 of \cite{AAD}, which is highly dependent on the dimensional and compactness assumptions. In certain cases we will extend this with a straightforward strategy which removes the compactness assumption. %Note there are a plethora of examples of complete, strictly almost K\"{a}hler manifolds (cf. \S4 of \cite{AAD1}).
%We recall the methodology of \cite{ST} to prove Corollary \ref{cor:2.5ofST}. by showing that static structures have $J$-invariant Ricci tensor, they can then apply the work of \cite{Armstrong} to consider a clean decomposition (in terms of self duality and two-form $J$ compatibility) of the Chern curvature in four dimensions which with \eqref{eq:staticdef} forces $\gw$ to be an eigenvector of the Weyl tensor. This setting  is equivalent to satisfying the \emph{third Gray condition} (cf. \cite{Gray} \S 5 (3)), so an application of \cite{AAD} Theorem 2 implies compact static structures are K\"{a}hler.
%We find that in comparison to the work of \cite{ST} the natural way to describe and study this flow is almost exclusively with respect to the Chern connection curvature. This follows in spirit with the more recent approach of Streets and Tian to pluriclosed flow, which was first introduced in \cite{ST10}. In \cite{ST1}, by noting the Bismut curvature acted as the flowing force of $\omega$, the authors demonstrated that rephrasing the flow in terms of the corresponding Bismut connection allowed one to glean off interesting facts about the flow's properties in a much more natural way. The key comparative observation is that in each case (K\"{a}hler Ricci, pluriclosed, and symplectic curvature flow) $\gw$ flows by the negative of some sort of scalar curvature of a given connection. The identity of this connection deeply influences the perspective with which one should approach the flow.
%
\begin{prop}\label{thm:rigidity} Suppose $\prs{M^4, \gw, J}$ is a complete almost K\"{a}hler manifold which is a \emph{static structure} in the sense of \eqref{eq:staticdef} for $\la \geq 0$. Suppose further that there is some $C_0 > 0$ such that
\begin{equation}\label{eq:naturalbds}
\Vol ( \cB_R ) \leq C_0 R^4, 
\qquad \brs{\brs{\Rm}}_{L^2}^2 \leq C_0.
\end{equation}
Then $\prs{M^4,\gw,J}$ is K\"{a}hler--Einstein.
\end{prop}
\begin{rmk} While we cannot make a statement for static points with $\la < 0$, we improved on the result of \cite{ST} by upgrading `compactness' assumption to `completeness'. This requires breaking away from using the third Gray condition as in \cite{ST} and directly analyzing the evolution of the $L^2$-norm of torsion.
\end{rmk}
\begin{rmk}
Note that for $\la \geq 0$ we have that $\Rc \geq 0$, so by Bishop--Gromov volume comparison theorem $\lim_{R \to \infty}\tfrac{\Vol(B_R)}{R^4}$ is monotonically decreasing. In particular, the imposed volume growth condition \eqref{eq:naturalbds} in the statement of Theorem \ref{thm:rigidity} agrees with this.
\end{rmk}
\subsection{Reformulation of Sekigawa's formula}\label{s:sekigawa}
We examine the key ingredient to the proof of Proposition \ref{thm:rigidity} called \emph{Sekigawa's formula}, which was first featured in \cite{Sekigawa} and further explored in \cite{ADM}, \cite{Kirchberg}, etc. On any almost K\"{a}hler manifold, by Chern--Weil theory the first Pontrjagin class has two representatives: $\mathfrak{p}_1 \brk{D}$ and $\mathfrak{p}_1 \brk{\N}$. Their difference is exact so by Stokes theorem:
\begin{equation*}
\int_M \prs{\mathfrak{p}_1 \brk{D} - \mathfrak{p}_1 \brk{\N} } \w \gw^{\w \prs{n-2}} = 0.
\end{equation*}
Sekigawa explicitly computed each term and from this derived an integral formula which he used to address the case of $\cR \geq 0$ in the Goldberg Conjecture (stating `A compact almost K\"{a}hler Einstein manifold is K\"{a}hler' in \cite{Goldberg}). We state and utilize the pointwise version featured in \cite{ADM}.
\begin{customprop}{1 of \cite{ADM}}\label{prop:SekigawaInt} For any almost K\"{a}hler manifold $\prs{M,J,\gw}$,
\begin{align}
\begin{split}\label{eq:ptwiseSekigawa}
\lap \prs{\cR - \cR^{\star}}
&=  \tfrac{1}{8}\brs{\phi}^2 + \tfrac{1}{2} \brs{\lap_D \gw}^2 + 2\brs{\cW^{2,0+0,2}}^2- 2 \brs{ \Rc^{2,0+0,2}}^2 \\
&\hsp - 2 g^{ia} g^{jb} \prs{\lap_D \gw_{ab} + \tfrac{1}{2} \phi_{ab}} \prs{ J_i^s \Rc^{1,1}_{js} }\\
&\hsp  - 4 g^{pi}g^{qa}D_i \brk{ J_p^j D_a \brk{J_q^b \Rc_{jb}^{2,0+0,2}}} - 2g^{uv} D_u \brk{\prs{D_v\gw^{ij}} \cQ_{ij} },
\end{split}
\end{align}
where $\cR^{\star} \triangleq \tfrac{1}{2} \gw^{ji} \gw^{kl} \Rm_{ijkl}$ is star-scalar curvature, $\phi \prs{X,Y} \triangleq \ip{D_{JX} \gw, D_Y \gw}$, $\lap_D$ is the Levi-Civita rough Laplacian, and $\cW$ is the Weyl tensor. 
\end{customprop}
\noindent While computationally nontrivial, the proof is a natural analysis of Bochner formulas. Note that the constants of the statement are modified to match with our conventions, in contrast to the constants of  \cite{ADM}. To prove Proposition \ref{thm:rigidity}, we convert \eqref{eq:ptwiseSekigawa} into a more suitable format for our purposes.
\begin{cor}\label{cor:topropseki} Equivalently for any almost K\"{a}hler manifold $\prs{M,g,J,\gw}$,
\begin{align*}
-2 \ip{\cB,\Rc^{1,1}}
&=  \brs{ \Rc^{2,0+0,2}}^2 - \brs{\cW^{2,0+0,2}}^2 -  \tfrac{1}{4} \brs{\cP^{2,0+0,2} }^2 -  \tfrac{1}{4} \brs{\cB^2 }^2 -  \tfrac{1}{16} \brs{\cB^1}^2 \\
&\hsp - g^{uv} D_u \brk{\prs{D_v\gw^{ij}} \cQ_{ij} } - 2 g^{pi}g^{qa}D_i \brk{ J_p^j D_a \brk{J_q^b \Rc_{jb}^{2,0+0,2}}} + \lap \brs{\tau}^2.
\end{align*}
To prove Corollary \ref{cor:topropseki}, we need to take a moment to convert the necessary parts of the original formula.
\begin{customlemma}{3.4 of \cite{ST}}[\textbf{AK}]\label{lem:Chernterms} We have that $\lap_D \gw_{ij}
%&= - 2 J_e^w \prs{ \N_w \tau^e_{ij}} - 4 \gw^{ue}\tau_{iuw}\tau^w_{je}\\
= - \cP_{ij}^{2,0+0,2} - 4 \gw^{ue}\tau_{iuw}\tau^w_{je}$.
\begin{rmk}
We note that we can further manipulate the second term to obtain
\begin{align*}
 - 4 \gw^{ue}\tau_{iuw}\tau^w_{je} &=   4 J_i^y J_u^v \gw^{ue}\tau_{yvw}\tau^w_{je} &\mbox{Corollary \ref{cor:tausymmetries}}\\
&=  - 4 J_i^y g_{wm}g^{ve}\tau_{yv}^m\tau^w_{je} \\
&=  - J_i^y \cB_{yj}^2. &\mbox{Lemma \ref{lem:Bformulas}}
\end{align*}
Therefore we have that
\begin{equation}\label{eq:lapgw}
\lap_D \gw_{ij} = - \cP_{ij}^{2,0+0,2} - J_i^s \cB_{sj}^2.
\end{equation}
\end{rmk}
\begin{comment}
\begin{proof}
We compute
%
\begin{align*}
\lap_D \gw_{kl}&= g^{ij} D_i D_j \gw_{kl}\\
&=2 g^{ij} D_i \prs{\gw_{je} \tau^e_{kl}} &\mbox{Corollary \ref{cor:N=DJ}}\\
&=2 \prs{ g^{ij} \N_i \prs{\gw_{je} \tau^e_{kl}} -  g^{ij} \Theta_{ik}^y \gw_{je} \tau^e_{yl} -  g^{ij} \Theta_{il}^y \gw_{je} \tau^e_{ky} }&\mbox{$D = \N + \Theta$}\\ 
&=- 2 J_e^i \prs{ \N_i \tau^e_{kl}} + 2\prs{J^e_l \Theta_{ik}^y- J^e_k \Theta_{il}^y}\tau^i_{ye}&\mbox{combine, Lemma \ref{cor:tausymmetries}}
\end{align*}
%
The first quantity is type $\prs{2,0+0,2}$. For the second quantity, which is type $(1,1)$, we have
%
\begin{align*}
2\prs{J^e_l \Theta_{ik}^y- J^e_k \Theta_{il}^y} \tau^i_{ye}
&= 2 \prs{J^e_l g^{yu} \tau_{kui}\tau^i_{ye}-J^e_k g^{yu} \tau_{lu}^i\tau_{yei}}&\mbox{expand $\Theta$}\\
&= 2 \prs{-\gw^{eu}\tau_{kui}\tau^i_{el} + \gw^{eu}\tau_{lu}^i\tau_{eki}}&\mbox{combine}\\ 
&=- 4 \gw^{eu}\tau^i_{el}\tau_{kui}. &\mbox{simplify}
\end{align*}
%
Therefore
%
\begin{align}
\begin{split}\label{lem:Chernterms}
\lap_D \gw_{kl}  & =\prs{\lap_D \gw}^{1,1}_{kl} + \prs{\lap_D \gw}^{2,0 + 0,2}_{kl} =\brk{-4 \gw^{eu}\tau^i_{el}\tau_{kui}} + \brk{- 2 J_e^i \prs{ \N_i \tau^e_{kl}} }.
\end{split}
\end{align}
%
The result follows.
\end{proof}
\end{comment}
\end{customlemma}
\begin{lemma}\label{lem:phiid} We have that
\begin{equation*}
 \phi_{ij}=4\gw^{mk} g^{ld} \tau_{mli}\tau_{kdj} = J_i^v \cB_{vj}^1.
\end{equation*}
\begin{proof}
We have that
\begin{align*}
\phi_{ij}
&= 4 \prs{g^{se} \gw_{iy} \tau^y_{le} }J_s^p\prs{g^{ld} \gw_{jw} \tau^w_{pd} } &\mbox{Proposition \ref{cor:N=DJ}}\\
&=4 g^{ld} \tau^y_{le} \tau^w_{pd} \gw^{pe}\gw_{iy}  \gw_{jw} \\%&\mbox{rearrange/simplify} \\
&=4 g^{ld} \prs{J^y_u \tau^u_{lm} J_e^m}\prs{J^w_q\tau^q_{kd} J_p^k}\gw^{pe}\gw_{iy}  \gw_{jw} &\mbox{Corollary \ref{cor:tausymmetries}} \\
%&=4 g^{ld} \tau^u_{lm}\tau^q_{kd}\prs{ J_p^k\gw^{pe}J_e^m}\prs{ J^y_u\gw_{iy}}\prs{J^w_q \gw_{jw}} &\mbox{rearrange}\\
%&=4 g^{ld} \tau^u_{lm}\tau^q_{kd}\gw^{km}g_{ui}g_{qj} &\mbox{simplify}\\
&=4 g^{ld} \tau_{mli}\tau_{kdj}\gw^{mk}\\ %\mbox{rearrange}\\
 &= -4 g^{ld} J_m^u J_i^v\tau_{ulv}\tau_{kdj}\gw^{mk} &\mbox{Corollary \ref{cor:tausymmetries}} \\
&= 4 g^{ld} g^{uk} J_i^v\tau_{ulv}\tau_{kdj} \\%&\mbox{simplify}\\
&= J_i^v \cB_{vj}^1. &\mbox{Lemma \ref{lem:Bformulas}}
\end{align*}
The result follows.
\end{proof}
\end{lemma}
\begin{proof}[Proof of Corollary \ref{cor:topropseki}]
First, by tracing through Lemma \ref{lem:Omegatraced} we have that
\[
\cR - \cR^{\star} = -2 \brs{\tau}^2. 
\]
Lemma \ref{lem:phiid} deals with the first term on the right hand side of \eqref{eq:ptwiseSekigawa}. For the second term we decompose by types using \eqref{eq:lapgw}. For the second row of \eqref{eq:ptwiseSekigawa} we combine Lemma \ref{lem:phiid} and \eqref{eq:lapgw}:
\begin{align*}
\lap_D \gw_{ab} + \tfrac{1}{2} \phi_{ab} &= \prs{ - \cP_{ab}^{2,0+0,2} - J_a^s \cB_{sb}^2} + \tfrac{1}{2}J_a^s \cB_{sb}^1 = - \cP_{ab}^{2,0+0,2} + 2 J_a^s\cB_{sb}. &\mbox{\eqref{eq:cBforms}}
\end{align*}
Thus it follows that the second row of \eqref{eq:ptwiseSekigawa} is equal to precisely
\begin{align*}
- 2 g^{ia} g^{jb} \prs{\lap_D \gw_{ab} + \tfrac{1}{2} \phi_{ab}} \prs{ J_i^s \Rc^{1,1}_{js} } &= 4\ip{\Rc^{1,1}, \cB}.
\end{align*}
Rearranging and dividing through by a factor of $2$ yields the result.
\end{proof}
\end{cor}
We next express a term in Corollary \ref{cor:topropseki} using Chern quantities to prepare to prove Proposition \ref{thm:rigidity}.
\begin{lemma}\label{lem:sekigawadiv2} We have that
\begin{align*}
\tfrac{1}{2} g^{uv} D_u \brk{\prs{D_v\gw^{ij}} \cQ_{ij} }  &= - 2 g^{rs}g^{ja} \N_u \brk{\tau_{as}^u\prs{    \N_m \tau_{rj}^m }}.
\end{align*}
\begin{proof}
First observe that
\begin{align*}
D_v \gw^{ij} = g^{ib} g^{ja} \prs{ D_v \gw_{ab}} &= 2 g^{ib} g^{ja} \gw_{vs} \tau^s_{ab} &\mbox{Proposition \ref{cor:N=DJ}} \\
&=2 g^{ib} g^{ja} J_v^e \tau_{abe}.%&\mbox{lower $\tau$ index}
\end{align*}
Therefore we have that
\begin{align*}
\tfrac{1}{2} \prs{D_v\gw^{ij}} \cQ_{ij}  &= g^{ib} g^{ja} J_v^e \tau_{abe} \cP^{2,0+0,2}_{ij} \\
&= 2 g^{ib} g^{ja} J_v^e \tau_{abe} \gw^{mn}\prs{ \N_m \tau_{ijn} } &\mbox{Lemma \ref{lem:P2,0+0,2}} \\
&=2 g^{ib} g^{ja} J_v^e\prs{-J_e^w J_b^s  \tau_{asw}} \gw^{mn}\prs{ - J_i^r J_n^l \N_m \tau_{rjl} } &\mbox{Corollary \ref{cor:tausymmetries}} \\
%&=2 \prs{J_i^rg^{ib} J_b^s }g^{ja}\prs{J_v^e J_e^w}\tau_{asw}\prs{J_n^l \gw^{mn}}\prs{    \N_m \tau_{rjl} }\\
&= - 2 g^{rs}g^{ja}\tau_{asv}\prs{    \N_m \tau_{rj}^m }.
\end{align*}
Differentiating once more and contracting an index yields the result.
%Therefore we have that
%%
%\begin{align*}
%\tfrac{1}{2} g^{uv} D_u \brk{\prs{D_v\gw^{ij}} \cQ_{ij}} &= - 2g^{uv} g^{rs}g^{ja} \N_u \brk{\tau_{asv}\prs{    \N_m \tau_{rj}^m }},
%%&= - 2 g^{rs}g^{ja}\prs{ \N_u \tau_{as}^u }\prs{    \N_m \tau_{rj}^m }  - 2 g^{rs}g^{ja} \tau_{as}^u \prs{ \N_u \N_m \tau_{rj}^m }.
%\end{align*}
%%
%yielding the result.
\end{proof}
\end{lemma}
\begin{proof}[Proof of Proposition \ref{thm:rigidity}] Since we are working on a setting that is potentially noncompact but complete, we let $\eta = \eta_R$ be a cutoff function satisfying
\begin{equation}\label{eq:etaassump}
\supp \eta = B_{2R}, \qquad \eta \equiv 1 \text{ on }B_R, \qquad \text{ thus } \brs{\N \eta} \leq \tfrac{C}{R}.
\end{equation}
%
%First we have that
%%
%\begin{align*}
%\tfrac{\del}{\del t} \brk{\dVg} &= \tfrac{1}{2} \prs{g^{ij} \tfrac{\del g_{ij}}{\del t}} \dVg \\
%&=  \varrho \dVg \\
%&= \gw^{ji} \cP_{ij} \dVg \\
%&= - \gw^{ji} \prs{\la\gw_{ij}} \dVg \\
%&= - \la \dVg.
%\end{align*}
%
We will denote the weighted $L^2$ norm using our cuttoff function (to sufficiently high power) by
\begin{equation}\label{eq:bumpstuff}
\dVe \triangleq \eta^p \dVg \qquad \brs{\brs{f}}_{L^2, \eta} \triangleq \int_M \brs{f}^2 \dVe.
\end{equation}
First we have that
\begin{equation}\label{eq:int|tau|2}
\tfrac{\del}{\del t}\brk{ \int_M \brs{\tau}^2 \dVe } = \int_M \tfrac{\del \brs{\tau}^2}{\del t} \dVe + \int_M \brs{\tau}^2 \tfrac{\del}{\del t}\brk{\dVe}.
\end{equation}
Since $\prs{M,g_t,J_t}$ is a static structure, it follows that $\Rc^{2,0+0,2} \equiv 0$, and $\cP^{2,0+0,2} \equiv 0$ (cf. Lemma 9.2 of \cite{ST}). Using Lemma \ref{lem:normtauvar}, and applying the fact that $\frac{\del J}{\del t} \equiv 0$ for static structures.
\begin{align*}
\tfrac{\del}{\del t} \brs{\tau}^2 &= \ip{\tfrac{\del g}{\del t}, \cB} - \tfrac{1}{2} \brs{ \tfrac{\del J}{\del t} }^2 +2 \N_p \brk{g^{ap}  \gw^{bw}\tau_{bae} \tfrac{\del J_w^e}{\del t}}\\
&= \ip{- \la g, \cB}\\ %&\mbox{$\tfrac{\del J}{\del t} \equiv 0$}\\
&= \la \brs{\tau}^2. &\mbox{Corollary \ref{cor:Btraces}}
\end{align*}
%
%Furthermore using Proposition \ref{prop:RChernconn},
%\begin{align*}
%\tfrac{\del \cR}{\del t} &=- \ip{ \tfrac{\del g}{\del t}, \Rc} - \lap \prs{\tr_g \tfrac{\del g}{\del t}} +   g^{ip} g^{jq}\prs{ \N_j \N_i \tfrac{\del g_{pq}}{\del t}}\\
%&\hsp - g^{dw}g^{qj}\N_j \brk{\tfrac{\del g_{pd}}{\del t} \tau_{qw}^p } \\
%&=\la \ip{ g, \Rc} + \la \lap \prs{\tr_g g} - \la   g^{ip} g^{jq}\prs{ \N_j \N_i g_{pq}}+ \la g^{dw}g^{qj}\N_j \brk{\tau_{qwd} }&\mbox{$\tfrac{\del g}{\del t} = - \la g$} \\
%&=\la \cR &\mbox{$\N g = 0$, Lemma \ref{lem:tracelesstau}}\\
%&=\la \prs{\tfrac{1}{2} \varrho - \brs{\tau}^2 }\\
%&=\la \prs{- \tfrac{1}{2} \la - \brs{\tau}^2 } &\mbox{$\varrho = - \la$} \\
%&=- \tfrac{1}{2} \la^2 - \la \brs{\tau}^2. &\mbox{expand}
%\end{align*}
%%
So updating \eqref{eq:int|tau|2},
\begin{align*}
\la \brs{\brs{ \tau }}^2_{L^2,\eta} + \int_M \brs{\tau}^2 \tfrac{\del}{\del t}\brk{\dVe} &= \tfrac{\del}{\del t} \brs{\brs{ \tau }}^2_{L^2,\eta}.
\end{align*}
We will examine the left hand side an alternate way. We have
\begin{align*}
-\tfrac{\del}{\del t} \brs{\tau}^2
&= 2 \ip{\Rc, \cB} - 2 \brs{\cB}^2 + \tfrac{1}{2} \brs{\tfrac{\del J}{\del t}}^2 + 8 \N_p \brk{g^{ap} g^{ub} \tau_{bae} g^{mr}\prs{\N_r \tau_{mu}^e}} &\mbox{Lemma \ref{lem:normtauvar}}\\
%%%%%%
&=  \tfrac{1}{16}\brs{\cB^1}^2 + \tfrac{1}{4} \brs{\cB^2}^2 - 2 \brs{\cB}^2  + \brs{\cW^{2,0+0,2}}^2\\
&\hsp +\tfrac{1}{2} \lap \brs{\tau}^2 + \N_u \brk{- 2g^{uv} g^{rs}g^{ja} \tau_{asv}\prs{    \N_m \tau_{rj}^m }}. &\mbox{Corollary \ref{cor:topropseki}, Lemma \ref{lem:sekigawadiv2}}
\end{align*}
We observe that, updating our constant $C$ at each step,
\begin{align*}
\int_M \N \prs{\N \tau \ast \tau} \eta^2 \dVg &\leq C \int_M \brs{\N \tau}\brs{\tau} \brs{\N \eta} \eta \dVg \\
&\leq C \brs{\brs{\tau}}_{L^4(A_R)} \brs{\brs{\N \tau}}_{L^2(A_R)} \brs{\brs{ \N \eta }}_{L^4(A_R)} &\mbox{H\"{o}lder's inequality}\\
&\leq C \brs{\brs{\tau}}_{L^4(A_R)} \brs{\brs{\N \tau}}_{L^2(A_R)} \prs{ \prs{\tfrac{C}{R} }^4 \Vol \prs{A_R} }^{1/4} &\mbox{\eqref{eq:etaassump}}\\
&\leq C \brs{\brs{\tau}}_{L^4(A_R)} \brs{\brs{\N \tau}}_{L^2(A_R)} &\mbox{\eqref{eq:naturalbds}} \\
&\leq C \brs{\brs{\Rm}}_{L^2(A_R)} &\mbox{Lemmata \ref{lem:|tau|by|Rm|},  \ref{lem:|Ntau|L2by|Rm|}}
\end{align*}
We clarify the usage of Lemma \ref{lem:|Ntau|L2by|Rm|}: take a partition of unity  $\sqg{\phi_i}$ covering $A_R$, and apply this lemma to each region designated by $\phi_i$. Taking the sum over the partition yields the desired estimate. Note by the assumed finiteness of $\brs{\brs{\Rm}}_{L^2}$ that $\lim_{R \to \infty} \brs{\brs{\Rm}}_{L^2 \prs{A_R}} \equiv 0$, thus this above quantity is essentially negligible. We observe that if $\dim_{\mathbb{C}} M =2$, then using \eqref{eq:Btermstau4} yields
\begin{align*}
\tfrac{1}{16}&\brs{\cB^1}^2 + \tfrac{1}{4} \brs{\cB^2}^2 - 2 \brs{\cB}^2
%&=\tfrac{1}{16}\brs{\cB^1}^2 + \tfrac{1}{4} \brs{\cB^2}^2 - 2 \brs{\tfrac{1}{4} \cB^1 - \tfrac{1}{2} \cB^2}^2 &\mbox{\eqref{eq:cBforms}}\\
%&=\tfrac{1}{16}\brs{\cB^1}^2 + \tfrac{1}{4} \brs{\cB^2}^2
% - 2\prs{\tfrac{1}{16} \brs{\cB^1}^2 - \tfrac{1}{4} \ip{\cB^1, \cB^2} + \tfrac{1}{4} \brs{\cB^2}^2} &\mbox{expand}\\
%%&=\tfrac{1}{16}\brs{\cB^1}^2 + \tfrac{1}{4} \brs{\cB^2}^2 -\tfrac{1}{8} \brs{\cB^1}^2 + \tfrac{1}{2} \ip{\cB^1, \cB^2} - \tfrac{1}{2} \brs{\cB^2}^2\\
%  &=-\tfrac{1}{16}\brs{\cB^1}^2 - \tfrac{1}{4} \brs{\cB^2}^2
% + \tfrac{1}{2} \ip{\cB^1, \cB^2} &\mbox{simplify} \\
%&=- \tfrac{1}{2} \brs{\tau}^4 - \brs{\tau}^4 + 2 \brs{\tau}^4 &\mbox{\eqref{eq:Btermstau4}}\\
= \tfrac{1}{2} \brs{\tau}^4.% &\mbox{simplify}
\end{align*}
It follows that
\begin{align*}
-\la \brs{\brs{ \tau }}^2_{L^2,\eta}& - \int_M \brs{\tau}^2 \tfrac{\del}{\del t}\brk{\dVe}\\ &= \tfrac{1}{2} \brs{\brs{\tau}}^4_{L^4,\eta} + \brs{\brs{\cW^{2,0+0,2}}}^2_{L^2,\eta}
+\int_M \N \prs{\N \tau \ast \tau} \dVe   - \int_M \brs{\tau}^2 \tfrac{\del}{\del t}\brk{\dVe}.
\end{align*}
Sending $R \to \infty$ and rearranging and simplifying yields that
\begin{align*}
0 &= \tfrac{1}{2} \brs{\brs{\tau}}^4_{L^4} + \brs{\brs{\cW^{2,0+0,2}}}^2_{L^2}+ \la \brs{\brs{ \tau }}^2_{L^2}.
\end{align*}
 Thus $\tau \equiv N \equiv 0$, and thus $M^4$ is K\"{a}hler. Furthermore, $\cB = 0$ and thus
 \begin{align*}
 - \la g &\equiv \tfrac{\del g}{\del t} = -2 \Rc + 2 \cB = -2 \Rc.
 \end{align*}
Thus the only static structures with $\la \geq 0$ of symplectic curvature flow are K\"{a}hler--Einstein.\end{proof}
\subsection{Bounds on torsion by curvature}
We provide Chern torsion bounds used for Proposition \ref{thm:rigidity}.
\begin{lemma}\label{lem:|tau|by|Rm|}
For an almost K\"{a}hler manifold, we have that
\begin{align*}
\sup_M \brs{\tau}^2 \leq C_{\dim_{\mathbb{C}}M } \brs{\brs{ \Rm }}_{C^0 \prs{g}}.
\end{align*}
\begin{proof}% Recall that $\cR^{\star} = \tfrac{1}{2} \gw^{ji} \gw^{kl} \Rm_{ijkl}$. This follows immediately by the fact that $2 \brs{\tau}^2 = \prs{\cR - \cR^*}$ (which can be obtained by tracing through Lemma \ref{lem:Omegatraced}).
We consider $\tau$ in the form of $DJ$ via Corollary \ref{cor:N=DJ}. Observe that
\begin{align}
\begin{split}\label{eq:lapD1}
0 = \lap_D 1 &= \lap_D \brs{J}^2 \\
&=2  \ip{\lap_D J, J} +  2 \brs{D J}^2\\ %&\mbox{distribute}\\
&=2  \ip{\lap_D J, J} +  8 \brs{\tau }^2. %&\mbox{Corollary \ref{cor:N=DJ}}
\end{split}
\end{align}
For the first term we compute out, using that $\gw$ is harmonic (thus $D \gw \equiv 0$, $^* \gw = 0$),
\begin{align}
\begin{split}\label{eq:lapDJ}
\lap_D J_a^b = g^{ij} D_i D_j J_a^b
&= g^{bc}g^{ij} D_i D_j \gw_{ac} \\
&=- g^{bc}g^{ij} D_i \prs{ D_a \gw_{cj} + D_c \gw_{ja} } \\
&=- g^{bc}g^{ij} \prs{ \brk{D_i , D_a} \gw_{cj} + \brk{D_i , D_c} \gw_{ja} } \\
&=- g^{bc}g^{ij} \prs{ -\Rm_{iac}^d \gw_{dj} - \Rm_{iaj}^d \gw_{cd} - \Rm_{icj}^d \gw_{da} - \Rm_{ica}^d \gw_{jd} } \\
%&=   \Rm_{iace} g^{ed} \gw_{dj}g^{bc}g^{ij} - \Rc_{ae}  g^{ed}  \gw_{cd}g^{bc} - \Rc_{ce} g^{ed}  \gw_{da}g^{bc}+ \Rm_{icae} g^{ed}  \gw_{jd}g^{bc}g^{ij} &\mbox{expand} \\
&=   \Rm_{iae}^b \gw^{ei} + \Rc_{a}^d  J_d^b+ \Rc_{e}^bJ_a^e+ \Rm_{aei}^b \gw^{ei}.
\end{split}
\end{align}
Therefore we have that
\begin{align*}
\ip{\lap_D J, J}  = g_{bu} g^{av} J_v^u\prs{\lap_D J_a^b} 
&= - J_b^a \prs{\lap_D J_a^b} \\
&=  - J_b^a \Rm_{iae}^b \gw^{ei}  - J_b^a \Rc_{a}^d  J_d^b - J_b^a\Rc_{e}^bJ_a^e - J_b^a \Rm_{aei}^b \gw^{ei} \\
&= 2 J_b^a \Rm_{aie}^b \gw^{ei}  + 2 \cR.
\end{align*}
Applying this to \eqref{eq:lapD1} and manipulating yields that
\begin{align*}
 \brs{\tau}^2 &= \tfrac{1}{2} J_b^a \Rm_{aic}^b \gw^{ei} +\tfrac{1}{2} \cR,
\end{align*}
from which the result follows.
\end{proof}
\end{lemma}
\begin{lemma}\label{lem:|Ntau|L2by|Rm|} Assume the bounds of \eqref{eq:naturalbds}. Suppose that $\phi = \phi_R$ is a cutoff function satisfying
\begin{equation}\label{eq:phiassump}
\supp \phi = B_{2R}, \qquad \phi \equiv 1 \text{ on }B_R,
\end{equation}
There exists some $C_{\dim_{\mathbb{C}} M} > 0$ such that
\begin{align*}
\brs{\brs{\N \tau}}_{L^2, \phi}^2 \leq C_{\dim_{\mathbb{C}}M} \prs{ \brs{\brs{\Rm}}^{1/2}_{L^2, \phi}   + 1} \brs{\brs{\Rm}}_{L^2, \phi}^{3/2}.
\end{align*}
\begin{proof}We first observe that
\begin{align*}
\N \tau &= D \tau + \Theta \ast \tau &\mbox{\eqref{eq:Cherndef}}\\
&= - \tfrac{1}{2} D \prs{J DJ} + \tau^{\ast 2} &\mbox{\eqref{cor:1.5}}\\
&= J \prs{D D J} + \prs{DJ}^{\ast 2} +  \tau^{\ast 2}.
\end{align*}
\noindent Keeping Proposition \ref{cor:N=DJ} in mind, it is sufficient to focus on analyzing the behavior of the first term. We compute out, first with pointwise terms,
\begin{align*}
\brs{D D J}^2 &= g_{ad}g^{bc}g^{im} g^{jn}\prs{D_i D_j J_b^a}\prs{D_m D_n J_c^d} \\
&= -g^{im} g^{jn}\prs{D_i D_j J_d^c }\prs{D_m D_n J_c^d} \\
&= -g^{im} g^{jn}\prs{\brk{D_i, D_j} J_d^c }\prs{D_m D_n J_c^d}  -g^{im} g^{jn}\prs{D_j D_i J_d^c }\prs{D_m D_n J_c^d} \\
&= A^1 + A^2.
\end{align*}
For the first term we integrate against $\dVf$, which is defined in the same manner of $\dVe$ of the proof of Proposition \ref{thm:rigidity}, (cf. \eqref{eq:etaassump}). Within this, we compute the commutator terms and apply the weighted  H\"{o}lder's inequality
\begin{align*}
\int_M A^1 \dVf &= -\int_M g^{im} g^{jn}\prs{\brk{D_i, D_j} J_d^c }\prs{D_m D_n J_c^d} \dVf\\
 &=\int_M g^{im} g^{jn}\prs{\Rm_{ijd}^s J_s^c - \Rm_{ijs}^c J_d^s }\prs{D_m D_n J_c^d} \dVf \\% &\mbox{commutation} \\
  &\leq C \int_M \brs{ \Rm \ast DDJ } \dVf \\%&\mbox{simplify} \\
&\leq \tfrac{C}{\ge} \brs{\brs{\Rm}}^2_{L^2, \phi} + \ge\brs{ \brs{D D J}}^2_{L^2, \phi}. %&\mbox{H\"{o}lder's inequality}
\end{align*}
We then integrate the second term, perform integration by parts and then Levi-Civita derivatives.
\begin{align*}
\int_M A_2 \dVf %&= -\int_M g^{im} g^{jn}\prs{D_j D_i J_d^c }\prs{D_m D_n J_c^d} \dVe \\
&= \int_M g^{im} g^{jn}\prs{D_i J_d^c }\prs{D_j D_m D_n J_c^d} \dVf \\
&\hsp +\int_M g^{im} g^{jn}\prs{D_i J_d^c }\prs{D_m D_n J_c^d} \prs{D_j\eta} \dVf \\%&\mbox{integrate by parts}\\
&= \int_M g^{im} g^{jn}\prs{D_i J_d^c }\prs{\brk{D_j ,D_m} D_n J_c^d} \dVf\\
&\hsp +  \int_M g^{im}\prs{D_i J_d^c }\prs{D_m \lap_D J_c^d} \dVf\\
&\hsp +\int_M g^{im} g^{jn}\prs{D_i J_d^c }\prs{D_m D_n J_c^d} \prs{D_j\phi} \dVf  \\ %&\mbox{commute $D$'s}\\
&= A_{21} + A_{22} + A_{23}.
\end{align*}
For the first term of $A_2$ we have
\begin{align*}
A_{21} %&= \int_M g^{im} g^{jn}\prs{D_i J_d^c }\prs{\brk{D_j ,D_m} D_n J_c^d} \dVe \\
%&=-  \int_M g^{im} g^{jn}\prs{D_i J_d^c }\prs{\Rm_{jmn}^s \prs{D_s J_c^d}+ \Rm_{jmc}^s \prs{D_n J_s^d} - \Rm_{jms}^d \prs{D_n J_c^s}} \dVe \\
&\leq C  \int_M \brs{\prs{DJ}^{\ast 2} \ast \Rm} \dVf \\%&\mbox{commutation formula}\\
&\leq C \int_M \brs{\tau}^2 \brs{\Rm} \dVf &\mbox{Proposition \ref{cor:N=DJ}} \\
&\leq C \brs{\brs{\Rm}}^2_{L^2, \phi}. &\mbox{Lemma \ref{lem:|tau|by|Rm|}}
\end{align*}
Next we compute $A_{22}$, applying integration by parts followed by a H\"{o}lder's inequality
\begin{align*}
A_{22} 
&= - \int_M \prs{\lap_D J_d^c }\prs{\lap_D J_c^d} \phi \dVf \\
&\hsp- \int_M g^{im}\prs{D_i J_d^c }\prs{\lap_D J_c^d} \prs{D_m\phi} \dVf \\
&\leq C \brs{\brs{\lap_D J }}_{L^2, \phi}^2 + C \int_M \brs{DJ}\brs{\lap_D J} \brs{D \phi} \dVf \\
&\leq C \brs{\brs{\Rm}}_{L^2, \phi}^2 + C \int_M \brs{\tau}\brs{\Rm} \brs{D \phi} \dVf &\mbox{\eqref{eq:lapDJ}, Proposition \ref{cor:N=DJ}}\\
&\leq C \brs{\brs{\Rm}}_{L^2, \phi}^2 + C \brs{\brs{\tau}}_{L^4(A_R)} \brs{\brs{\Rm}}_{L^2(A_R)} \brs{\brs{D \phi}}_{L^4(A_R)} \\
&\leq C  \brs{\brs{\Rm}}_{L^2, \phi}^2 + C\brs{\brs{\tau}}_{L^4(A_R)} \brs{\brs{\Rm}}_{L^2(A_R)} \prs{ \prs{\tfrac{C_0}{R} }^4 \Vol \prs{A_R} }^{1/4} &\mbox{\eqref{eq:phiassump},\eqref{eq:naturalbds}}\\
&\leq C  \brs{\brs{\Rm}}_{L^2, \phi}^2 + C \brs{\brs{\Rm}}_{L^2(A_R)}^{3/2}. &\mbox{Lemma \ref{lem:|tau|by|Rm|}}
\end{align*}
Lastly we have, applying H\"{o}lder's inequality twice
\begin{align*}
A_{23} &\leq C \int_M \brs{\prs{D J} \ast \prs{DDJ} \ast D\phi }\dVe \\ %&\mbox{simplify} \\
&\leq C \brs{\brs{\tau}}_{L^4(A_R)} \brs{\brs{DDJ}}_{L^2(A_R)} \brs{\brs{D \phi}}_{L^4(A_R)}\\
&\leq C \brs{\brs{\tau}}_{L^4(A_R)} \brs{\brs{DDJ}}_{L^2(A_R)} \prs{ \prs{\tfrac{C_0}{R} }^4 \Vol \prs{A_R} }^{1/4} &\mbox{\eqref{eq:phiassump},\eqref{eq:naturalbds}}\\
&\leq \tfrac{C}{\ge} \brs{\brs{\tau}}_{L^4(A_R)}^2 +  \ge \brs{\brs{DDJ}}_{L^2(A_R)}^2 \\
&\leq \tfrac{C}{\ge} \brs{\brs{\Rm}}_{L^2(A_R)} +  \ge \brs{\brs{DDJ}}_{L^2(A_R)}^2. &\mbox{Lemma \ref{lem:|tau|by|Rm|}}
\end{align*}
Combining these various estimates together it follows that
\begin{align*}
\brs{\brs{D D J}}^2_{L^2, \phi} &\leq \tfrac{C}{\ge} \brs{\brs{\Rm}}^2_{L^2, \phi} + \ge\brs{ \brs{D D J}}^2_{L^2, \phi} + C \brs{\brs{\Rm}}^2_{L^2, \phi} + C  \brs{\brs{\Rm}}_{L^2, \phi}^2 + C \brs{\brs{\Rm}}_{L^2(A_R)}^{3/2} \\
&\hsp + \tfrac{C}{\ge} \brs{\brs{\Rm}}_{L^2(A_R)} +  \ge \brs{\brs{DDJ}}_{L^2(A_R)}^2 \\
&\leq C \prs{ \brs{\brs{\Rm}}^{1/2}_{L^2, \phi}   + 1} \brs{\brs{\Rm}}_{L^2, \phi}^{3/2} + 2\ge\brs{ \brs{D D J}}^2_{L^2, \phi}.
\end{align*}
The result follows.
\end{proof}
\end{lemma}
\bibliography{sources.bib}{}
\bibliographystyle{alpha}
\end{document}